\documentclass{statsoc}

\usepackage{amsmath, amssymb}
\usepackage{geometry} 
\usepackage{vmargin} 
\usepackage{natbib}
\usepackage{color}
\usepackage{float}
\usepackage{graphicx}
\allowdisplaybreaks[4]

\newtheorem{thm}{Theorem}[section]
\newtheorem{lem}[thm]{Lemma}

\newtheorem{defin}[thm]{Definition}
\newtheorem{rmk}[thm]{Remark}
\newtheorem{cor}[thm]{Corollary}

\newcommand{\kc}{\mathcal{C}}

\newcommand{\pc}{\mathcal{P}_\mathcal{C}}

\DeclareMathOperator{\inv}{inv}
\DeclareMathOperator{\invs}{invsum}

\title[On the exact region determined by Kendall's tau and \mbox{Spearman's rho}]{On the exact 
region determined by Kendall's tau and \\ \mbox{Spearman's rho}}
\author[Manuela Schreyer {\it et al.}]{Manuela Schreyer}
\email{manuelalarissa.schreyer@sbg.ac.at}

\author{Roland Paulin\thanks{Supported by the Austrian Science Fund (FWF): P24574}}
\author[Schreyer, Paulin, Trutschnig]{Wolfgang Trutschnig\thanks{Corresponding author, 
Hellbrunnerstr. 34, A-5020 Salzburg, wolfgang.trutschnig@sbg.ac.at}}
\address{University of Salzburg, Salzburg, Austria.}

\begin{document}

\begin{abstract}
Using properties of shuffles of copulas and tools from combinatorics we solve the open question about the exact region 
$\Omega$ determined by all possible values of Kendall's $\tau$ and Spearman's $\rho$. 
In particular, we prove that the well-known inequality established by Durbin and Stuart in 1951 is not sharp outside 
a countable set, give a simple analytic characterization of $\Omega$ in terms of 
a continuous, strictly increasing piecewise concave function, and show that 
$\Omega$ is compact and simply connected, but not convex. 
The results also show that for each $(x,y)\in \Omega$ there are mutually completely dependent random variables $X,Y$
whose $\tau$ and $\rho$ values coincide with $x$ and $y$ respectively. 
\end{abstract}

\section{Introduction}
Kendall's $\tau$ and Spearman's $\rho$ are, without doubt, the two most famous nonparametric measures of concordance.
Given random variables $X,Y$ with continuous distribution functions $F$ and $G$ respectively, Spearman's $\rho$
is defined as the Pearson correlation coefficient of the $\mathcal{U}(0,1)$-distributed random 
variables $U:=F \circ X$ and $V:=G \circ Y$ whereas Kendall's $\tau$ is given by the probability of concordance minus the 
probability of discordance, i.e. 
\begin{align*}
\rho(X,Y)&= 12\big(\mathbb{E}(UV)-\tfrac{1}{4}\big) \\
\tau(X,Y)&= \mathbb{P}\big((X_1-X_2)(Y_1 -Y_2)>0) - \mathbb{P}\big((X_1-X_2)(Y_1 -Y_2)<0\big),
\end{align*}
where $(X_1,Y_1)$ and $(X_2,Y_2)$ are independent and have the same distribution as $(X,Y)$. 
Since both measures are scale invariant they only depend on the underlying (uniquely determined) copula $A$ of $(X,Y)$.
It is well known and straightforward to verify \citep{Ne} that, given the copula $A$ of $(X,Y)$, 
Kendall's $\tau$ and Spearman's $\rho$ can be expressed as
\begin{align}
\tau(X,Y)&= 4 \int_{[0,1]^2} A(x,y)\, d\mu_A(x,y) - 1=:\tau(A)  \label{deftau} \\
\rho(X,Y)&= 12 \int_{[0,1]^2} xy \,d\mu_A(x,y) - 3=:\rho(A) \label{defrho}, 
\end{align}
where $\mu_A$ denotes the doubly stochastic measure corresponding to $A$. Considering that 
$\tau$ and $\rho$ quantify different aspects of the underlying dependence structure 
 \citep{FreNe} a very natural question is how much they can 
differ, i.e. if $\tau(X,Y)$ is known which values may $\rho(X,Y)$ assume and vice versa. 
The first of the following two well-known universal inequalities
between $\tau$ and $\rho$ goes back to \cite{Dan}, the second one to \cite{D+S} (for alternative proofs see \cite{Kr,Ge,Ne}):  
\begin{equation}\label{ineqdaniels}
\vert 3 \tau-2 \rho \vert \leq 1
\end{equation}
\begin{equation}\label{ineqdurbinstuart}
\frac{(1+\tau)^2}{2} -1 \leq \rho \leq 1- \frac{(1-\tau)^2}{2}
\end{equation}
The inequalities together yield the set $\Omega_0$ (see Figure \ref{regionintro}) 
which we will refer to as \emph{classical $\tau$-$\rho$ region} in the sequel.
Daniels' inequality is known to be sharp \citep{Ne} whereas the first part of the inequality by Durbin and Stuart
is only known to be sharp at the points 
$\boldsymbol{p}_n=(-1+\frac{2}{n},-1+\frac{2}{n^2})$ with $n \geq 2$ (which, using symmetry, is
to say that the second part is sharp at the points $-\boldsymbol{p}_n$). 
Although both inequalities are known since the 1950s and the interrelation between $\tau$ and $\rho$ has 
received much attention also in recent years, in particular concerning the so-called
Hutchinson-Lai conjecture \citep{FreNe,Hue,BL}, 
to the best of the authors' knowledge the \emph{exact $\tau$-$\rho$ region $\Omega$}, defined
by ($\kc$ denoting the family of all two-dimensional copulas)
\begin{align}\label{defregionall}
\Omega&=\big\{(\tau(X,Y),\rho(X,Y)):\,\, X,Y \textrm{ continuous random variables}\big\} \\
 &= \big\{(\tau(A),\rho(A)):\,\, A \in \kc\big\} \nonumber,
\end{align} 
is still unknown.

In this paper we solve the sixty year old question and give a full characterization of $\Omega$. 
We derive a piecewise concave, strictly increasing, continuous 
function $\Phi:[-1,1]\rightarrow [-1,1]$ and (see Theorem \ref{mainres} and Theorem \ref{identical}) prove that 
\begin{align}\label{finalobj}
\Omega=\big\{(x,y)\in [-1,1]^2: \Phi(x)\leq y \leq -\Phi(-x) \big\}.
\end{align} 
Figure \ref{regionintro} depicts $\Omega_0$ and the function $\Phi$ (lower red line), the explicit form of 
$\Phi$ is given in eq. (\ref{Phim}) and eq. (\ref{Phi}).   
As a byproduct we get that the inequality by Durbin and Stuart is not sharp outside the aforementioned points 
$\boldsymbol{p_n}$ and $\boldsymbol{-p_n}$, 
that $\Omega$ is compact and simply connected, but not convex. Moreover, we prove the surprising fact that
for each point $(x,y) \in \Omega$ there exist mutually completely dependent random variables $X,Y$ for which
$(\tau(X,Y),\rho(X,Y))=(x,y)$ holds.  
\begin{figure}[H]
\centering
\makebox{\includegraphics[width=10cm]{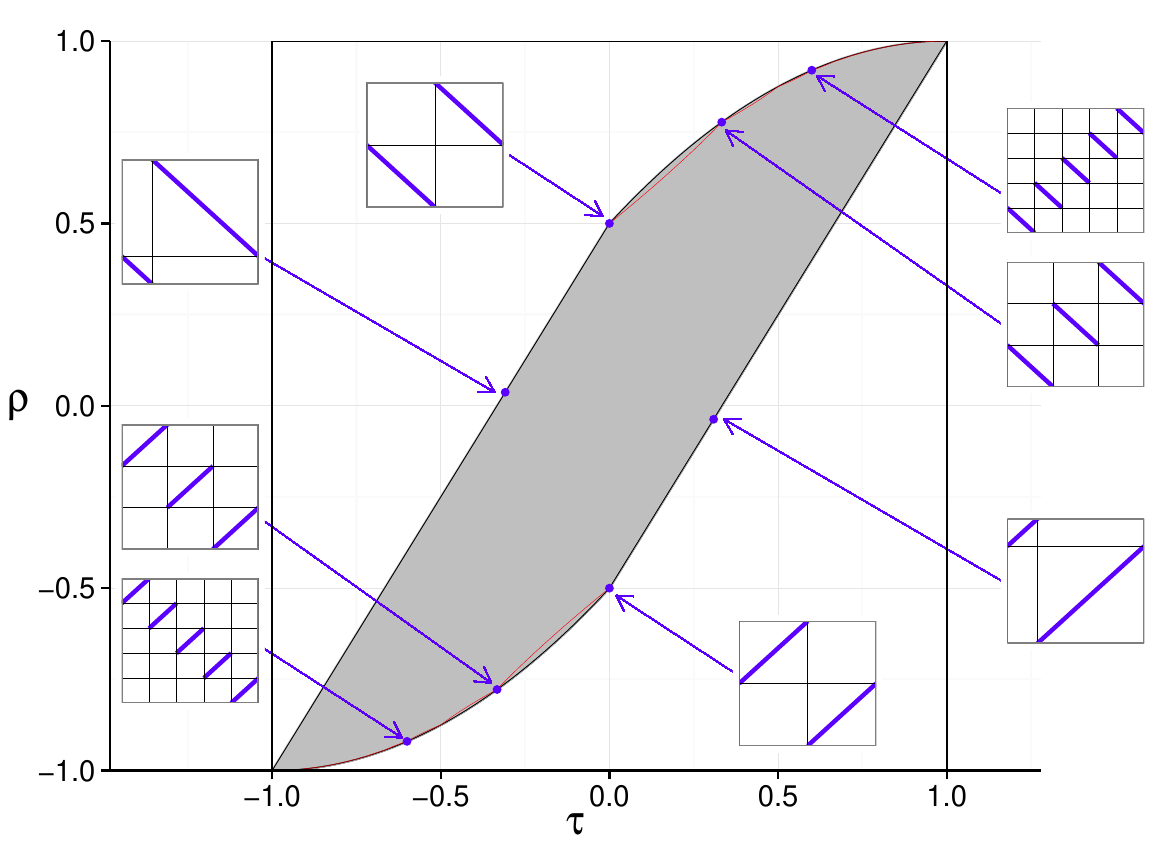}}\caption{The classical $\tau$-$\rho$-region $\Omega_0$ 
and some copulas (distributing mass uniformly on the blue segments)
for which the inequality by Durbin and Stuart is sharp. The red line depicts the true boundary of $\Omega$.}
\label{regionintro}
\end{figure}

The rest of the paper is organized as follows: Section 2 gathers some notations and preliminaries. In Section 
3 we reduce the problem of determining $\Omega$ to a problem about so-called shuffles of copulas, prove
some properties of shuffles and derive the function $\Phi$. The main result saying that 
$\Omega$ is contained in the right-hand-side of eq. (\ref{finalobj}) is given in Section 4, tedious 
calculations needed for the proofs being collected in the Appendix. Section 5 serves to 
prove equality in eq. (\ref{finalobj}) and to collect some interesting consequences of this result. Finally, in Section 6 we gather 
some new conjectures on the exact $\tau$-$\rho$ region for well-known subclasses of copulas.  
  
\section{Notation and Preliminaries}
As already mentioned before, $\kc$ will denote the family of all two-dimen\-sional \emph{copulas},
see \cite{DuS,Em,Ne}.
$M$ and $W$ will denote the upper and the lower Fr\'echet-Hoeffding bound respectively.  
Given $A\in \kc$ the \emph{transpose} $A^t \in \kc$ of $A$ is defined by $A^t(x,y):=A(y,x)$ for all $x,y \in [0,1]$.
$d_\infty$ will denote the uniform distance on $\kc$; it is well known that $(\kc,d_\infty)$ is a compact metric space and
that $d_\infty$ is a metrization of weak convergence in $\kc$. 
For every $A \in \kc$ the corresponding \emph{doubly stochastic measure} will be denoted by $\mu_A$, i.e. we have
$\mu_A([0,x]\times [0,y]):=A(x,y)$ for all $x,y\in [0,1]$. $\mathcal{P}_\mathcal{C}$ denotes the class
of all these doubly stochastic measures.
$\mathcal{B}([0,1])$ and $\mathcal{B}([0,1]^2)$ will denote the Borel $\sigma$-fields in $[0,1]$ and $[0,1]^2$, 
$\lambda$ and $\lambda_2$ the Lebesgue measure on $\mathcal{B}([0,1])$ and $\mathcal{B}([0,1]^2)$ respectively. 
Instead of $\lambda$-a.e. we will simply write a.e. since no confusion will arise. 
$\mathcal{T}$ will denote the class of all $\lambda$-preserving transformations $h:[0,1] \rightarrow [0,1]$, i.e.
transformations for which the push-forward $\lambda^h$ of $\lambda$ via $h$ coincides with $\lambda$,  
$\mathcal{T}_b$ the subclass of all bijective $h \in \mathcal{T}$. 

For every copula $A \in \kc$ there exists a \emph{Markov kernel} (regular conditional distribution) 
$K_A:[0,1] \times \mathcal{B}([0,1]) \rightarrow [0,1]$ fulfilling 
($G_x:=\{y \in [0,1]:(x,y) \in G\}$ denoting the $x$-section of $G \in \mathcal{B}([0,1]^2)$ for 
every $x \in [0,1]$)
\begin{equation}\label{bound}
 \int_{[0,1]} K_A(x,G_x)\, d\lambda(x) = \mu_A(G),
\end{equation}
for every $G \in \mathcal{B}([0,1]^2)$, so, in particular
 \begin{equation}\label{bound2}
 \int_{[0,1]} K_A(x,F)\, d\lambda(x) = \lambda(F)
\end{equation}
for every $F \in \mathcal{B}([0,1])$, see \cite{Tru}. We will refer to $K_A$ simply as Markov kernel of $A$.
On the other hand, every Markov kernel $K:[0,1]\times \mathcal{B}([0,1]) \rightarrow [0,1]$ fulfilling (\ref{bound2}) 
induces a unique element $\mu \in \pc([0,1]^2)$ via (\ref{bound}). 
For more details and properties of disintegration we refer to \citep{Ka,Kl}.\\
A copula $A\in \kc$ will be called \emph{completely dependent} if and only if there exists 
$h \in \mathcal{T}$ such that $K(x,E):=\mathbf{1}_E(h(x))$ is a Markov kernel of $A$ 
(see \cite{Tru} for equivalent definitions and main properties). 
For every $h \in \mathcal{T}$ the induced completely dependent copula will be denoted by $A_h$.
Note that $h_1=h_2$ a.e. implies $A_{h_1}=A_{h_2}$ and that eq. (\ref{bound}) implies 
$A_h(x,y)=\lambda([0,x] \cap h^{-1}([0,y]))$ for all $x,y \in [0,1]$. In the sequel $\kc_d$ will 
denote the family of all completely dependent copulas. 
$A_h \in \kc_d$ will be called \emph{mutually completely dependent} if we even have $h \in \mathcal{T}_b$. Note that
in case of $h \in \mathcal{T}_b$ we have $A_{h^{-1}}=(A_h)^t$.
Complete dependence is the opposite of independence since it describes the (not necessarily mutual) 
situation of full predictability/maximum dependence. 

Tackling the problem of determining the region $\Omega$, our main tool will be special  
members of the class $\kc_d$ usually
referred to as \emph{shuffles of} the minimum copula \emph{$M$}. 
Following \citep{Ne} we will call $h\in \mathcal{T}_b$ a
\emph{shuffle} (and $A_h \in \kc_d$ a shuffle of $M$) if there exist $0=s_0<s_1<\ldots < s_{n-1} < s_n=1$ and 
$\boldsymbol{\varepsilon}=(\varepsilon_1,\ldots,\varepsilon_n) \in \{-1,1\}^n$ such that we have 
$h'(x)=\varepsilon_i$ for every $x \in (s_{i-1},s_i)$. In case of $\varepsilon_i=1$ for every $i \in \{1,\ldots,n\}$ we will
call $h$ \emph{straight shuffle}. $\mathcal{S}$ will denote the family of all shuffles, 
$\mathcal{S}^+$ the family of all straight shuffles. It is well known \citep{Mish,Ne} that 
$\kc_{\mathcal{S}^+}$, defined by
\begin{align}
\kc_{\mathcal{S}^+}=\big\{A_h: \,h\in \mathcal{S}^+\big\}
\end{align}
is dense in $(\kc,d_\infty)$. For more general definitions of shuffles we refer 
to \citep{DuS}.  
Obviously every shuffle $h \in \mathcal{S}$ can be expressed in terms
of vectors $\boldsymbol{u} \in \Delta_n, \boldsymbol{\varepsilon} \in \{-1,1\}^n$ and a permutation 
$\boldsymbol{\pi} \in \sigma_n$, where 
$\Delta_n$ denotes the unit simplex $\Delta_n=\{\boldsymbol{x} \in [0,1]^n: \sum_{i=1}^n x_i=1\}$ and $\sigma_n$ denotes
all bijections on $\{1,\ldots,n\}$. In fact, choosing suitable
$\boldsymbol{u} \in \Delta_n, \boldsymbol{\varepsilon} \in \{-1,1\}^n, \boldsymbol{\pi} \in \sigma_n$, setting 
(empty sums are zero by definition)
\begin{eqnarray}
s_k&:=&\sum_{i=1}^k u_i, \qquad t_k:=\sum_{i=1}^k u_{\boldsymbol{\pi}^{-1}(i)}
\end{eqnarray}
for every $k \in \{0,\ldots,n\}$, we have $s_k-s_{k-1}=u_k=t_{\boldsymbol{\pi}(k)}-t_{\boldsymbol{\pi}(k)-1}$ 
and on $(s_{k-1},s_k)$ the shuffle $h$ is given by   
\begin{equation}\label{defshuffle}
h(x) = h_{\boldsymbol{\pi},\boldsymbol{u},\boldsymbol{\varepsilon}}(x):= \left\{
\begin{array}{ll}
t_{\boldsymbol{\pi}(k)-1} + x - s_{k-1} & \text{if } \varepsilon_k=1,\\
t_{\boldsymbol{\pi}(k)} - (x-s_{k-1}) & \text{if } \varepsilon_k=-1.
\end{array} \right.
\end{equation}
In the sequel we will work directly with the function $h_{\boldsymbol{\pi},\boldsymbol{u},\boldsymbol{\varepsilon}}$, implicitly defined in eq. (\ref{defshuffle})
since all possible extensions of $h_{\boldsymbol{\pi},\boldsymbol{u},\boldsymbol{\varepsilon}}$ from $\bigcup_{i=1}^k(s_{k-1},s_k)$ to $[0,1]$ yield the same copula,
which we will denote by $A_{h_{\boldsymbol{\pi},\boldsymbol{u},\boldsymbol{\varepsilon}}}$. In case of $\varepsilon_i=1$ for every $i \in \{1,\ldots,n\}$ 
we will simply write $h_{\boldsymbol{\pi},\boldsymbol{u}}$ in the sequel. Note that the chosen representation is not unique, i.e. for given
$\boldsymbol{u} \in \Delta_n, \boldsymbol{\varepsilon} \in \{-1,1\}^n, \boldsymbol{\pi} \in \sigma_n$ there always exist 
$\boldsymbol{u}' \in \Delta_m, \boldsymbol{\varepsilon}' \in \{-1,1\}^m, \boldsymbol{\pi}' \in \sigma_m$ with $m \not = n$ such that 
$h_{\boldsymbol{\pi},\boldsymbol{u},\boldsymbol{\varepsilon}}=h_{\boldsymbol{\pi}',\boldsymbol{u}',\boldsymbol{\varepsilon}'}$ a.e., 
implying $A_{h_{\boldsymbol{\pi},\boldsymbol{u},\boldsymbol{\varepsilon}}}=A_{h_{\boldsymbol{\pi}',\boldsymbol{u}',\boldsymbol{\varepsilon}'}}$. 
So, for instance, the shuffle $h_{\boldsymbol{\pi},\boldsymbol{u},\boldsymbol{\varepsilon}}$ with 
$\boldsymbol{\pi}=(4,2,1,3), \boldsymbol{u}=(\frac{1}{8},\frac{3}{8},\frac{1}{4},\frac{1}{4})$ and 
$\boldsymbol{\varepsilon}=(1,-1,1,1)$ 
and the shuffle
$h_{\boldsymbol{\pi}',\boldsymbol{u}',\boldsymbol{\varepsilon}'}$ with 
$\boldsymbol{\pi}'=(5,3,1,2,4), \boldsymbol{u}'=(\frac{1}{8},\frac{3}{8},\frac{1}{8},\frac{1}{8},\frac{1}{4})$ 
and $\boldsymbol{\varepsilon}'=(1,-1,1,1,1)$ coincide a.e. and induce the same copula. 

\begin{rmk}   
\emph{It might seem more natural to work directly with minimal representations (minimal dimension $n$) and 
to exclude the case of $u_k=0$ for some $k$ (implying $(s_{k-1},s_k)=\emptyset$) in the first place -- since we 
will, however, use various compactness arguments in the sequel the chosen representation is more convenient.}
\end{rmk}

\section{Basic properties of $\Omega$ and some results on shuffles}
In this section we will first show that for determining $\Omega$ it is sufficient to consider straight shuffles,
give explicit formulas for $(\tau(A_h),\rho(A_h))$ for arbitrary $h \in \mathcal{S}^+$, and derive a strictly increasing  
function $\Phi:[-1,1] \rightarrow [-1,1]$ which, after some change of coordinates, will finally be shown to 
determine $\Omega$ in the subsequent section. \\  
We start with some observations about $\Omega$. Considering that the mapping $f:\kc \rightarrow [-1,1]^2$, defined by
$f(A)=(\tau(A),\rho(A))$, is continuous w.r.t. $d_\infty$ \citep{Sca}, the compactness 
of $(\kc,d_\infty)$ implies the compactness of $\Omega$. As a consequence, using eq. (\ref{defregionall}) and the fact that 
$\kc_{\mathcal{S}^+}$ is dense we immediately get 
($\overline{U}$ denoting the closure of a set $U$) 
\begin{equation}\label{omegashuffle}
\Omega=\overline{\big\{(\tau(A_h),\rho(A_h)): \, h \in \mathcal{S}^+ \big\}}.
\end{equation}
Based on this, our method of proof will be to construct a compact set $\Omega_\Phi$ (fully determined by the function $\Phi$) 
fulfilling $(\tau(A_h),\rho(A_h)) \in \Omega_\Phi$ for every $h \in \mathcal{S}^+$ since we then 
automatically get $\Omega \subseteq \Omega_\Phi$.  

Being concordance measures, $\tau$ and $\rho$ fulfil the axioms mentioned in \cite{Sca}, so $\Omega$ is also symmetric w.r.t. $(0,0)$.  
Analogously, it is straightforward to verify that 
$\tau(A^t)=\tau(A)$ as well as $\rho(A^t)=\rho(A)$ holds for every $A \in \kc$, implying
\begin{equation}\label{transposeinvar}
\tau(A_{h^{-1}})=\tau(A_h), \quad \rho(A_{h^{-1}})=\rho(A_h)
\end{equation}  
for every $h \in \mathcal{T}_b$. 

For every $h \in \mathcal{T}$ define the quantities $\inv(h)$ and $\invs(h)$ (notation loosely based 
on \cite{SaUl}) by
\begin{eqnarray}
\inv(h) &=& \int_{[0,1]^2} \mathbf{1}_{[0,x)}(y) \mathbf{1}_{(h(x),1]}(h(y)) \,d\lambda_2(x,y) \label{definv} \\
\invs(h) &=& \int_{[0,1]^2} \mathbf{1}_{[0,x)}(y) \mathbf{1}_{(h(x),1]}(h(y))(x-y) \,d\lambda_2(x,y). \label{definvsum} 
\end{eqnarray}
\begin{lem}\label{trafo1}
For every $h \in \mathcal{T}_b$ the following relations hold:
\begin{eqnarray*}
\tau(A_h)&=& 4 \int_{[0,1]} A_h(x,h(x)) \,d\lambda(x) -1 = 1 - 4 \inv(h) \label{invtau}\\
\rho(A_h)&=& 12 \int_{[0,1]} xh(x) \,d \lambda(x) -3  = 1 - 12 \invs(h) \label{invsumrho}
\end{eqnarray*}
Moreover, for every $h \in \mathcal{T}_b$ we have 
$(\inv(h),\invs(h)) \in \big[0,\frac{1}{2}\big]\times \big[0,\frac{1}{6}\big]$.
\end{lem}
\begin{proof}
Using disintegration we immediately get
\begin{eqnarray*}
\tau(A_h)&=&4\int_{[0,1]} \int_{[0,1]} A_h(x,y) K_{A_h}(x,dy)\,d \lambda(x)-1  4 \int_{[0,1]} A_h(x,h(x)) \,d\lambda(x) -1
\end{eqnarray*}
as well as 
\begin{eqnarray*}
\inv(h)&=& \int_{[0,1]}  d\lambda(x) \int_{[0,x]} \Big( 1- \mathbf{1}_{[0,h(x)]}(h(y)) \Big) \,d\lambda(y) \\
         &=& \int_{[0,1]} \big(x - A_h(x,h(x)) \big) \,d\lambda(x) = \frac{1-\tau(A_h)}{4}
\end{eqnarray*}
which proves the first identity. The first part of the second one is an immediate consequence of disintegration. To prove
the remaining equality use $\int_{[0,x)} \mathbf{1}_{(h(x),1]} (h(y))\, d\lambda(y)= x-A_h(x,h(x))$ and
$\int_{(y,1]} \mathbf{1}_{[0,h(y))} (h(x))\, d\lambda(x)= h(y)-A_h(y,h(y))$ to get
\begin{eqnarray*}
\invs(h)&=& \int_{[0,1]} x \big(x- A_h(x,h(x)) -(h(x)-A_h(x,h(x))  \big)\,d\lambda(x) \\
          &=& \frac{1}{3} - \int_{[0,1]} xh(x)\,d\lambda(x).
\end{eqnarray*} 
The fact that $(\inv(h),\invs(h)) \in \big[0,\frac{1}{2}\big]\times \big[0,\frac{1}{6}\big]$ is a direct
consequence of $\Omega \subseteq [-1,1]^2$.
\end{proof}
As next step we derive explicit formulas for $\inv(h)$ and $\invs(h)$ for the case of $h$ being a straight shuffle 
based on which we will afterwards derive 
the afore-mentioned function $\Phi$ determining the region $\Omega$. To simplify notation define
\begin{eqnarray}\label{defIQ}
I_{\boldsymbol{\pi}} &=& \big\{\{i,j\}: \, 1 \le i < j \le n, \, \boldsymbol{\pi}(i) > \boldsymbol{\pi}(j)\big\} \nonumber \\
Q_{\boldsymbol{\pi}} &=& \big\{\{i,j,k\}: \, 1 \le i < j < k \le n, \, \boldsymbol{\pi}(i) > \boldsymbol{\pi}(j) > \boldsymbol{\pi}(k) \textrm{ or } \\
  & & \qquad \boldsymbol{\pi}(j) > \boldsymbol{\pi}(k) > \boldsymbol{\pi}(i)  \textrm{ or } \boldsymbol{\pi}(k) > \boldsymbol{\pi}(i) > \boldsymbol{\pi}(j) \big\}, \nonumber 
\end{eqnarray}
as well as 
\begin{equation}\label{defab}
a_{\boldsymbol{\pi}}(\boldsymbol{u})= \inv(h_{\boldsymbol{\pi},\boldsymbol{u}}), \qquad b_{\boldsymbol{\pi}}(\boldsymbol{u})= \inv(h_{\boldsymbol{\pi}, \boldsymbol{u}}) - 2\invs(h_{\boldsymbol{\pi}, \boldsymbol{u}})
\end{equation}
for every $\boldsymbol{\pi} \in \sigma_n$ and $\boldsymbol{u} \in \Delta_n$. The following lemma (the proof of which is given in the Appendix) holds.
\begin{lem}\label{invexplicit}
For every $(\boldsymbol{\pi},\boldsymbol{u}) \in \sigma_n\times\Delta_n$ the following identities hold: 
\begin{eqnarray*}
\inv(h_{\boldsymbol{\pi}, \boldsymbol{u}}) &=& a_{\boldsymbol{\pi}}(\boldsymbol{u}) = \sum_{i<j, \, \{i,j\} \in I_{\boldsymbol{\pi}}} u_i u_j \\
\invs(h_{\boldsymbol{\pi}, \boldsymbol{u}}) &=& \sum_{i<j, \,\{i,j\} \in I_{\boldsymbol{\pi}}} \left(\frac{1}{2} u_i^2 u_j + 
\frac{1}{2} u_i u_j^2 + \sum_{k: \, i<k<j} u_i u_j u_k \right) \\
b_{\boldsymbol{\pi}}(\boldsymbol{u}) &=& \sum_{i<j<k, \, \{i,j,k\} \in Q_{\boldsymbol{\pi}}} u_i u_j u_k
\end{eqnarray*}
\end{lem}  
\begin{rmk}
\emph{Notice that in \citep[Propositions 4 and 5]{Ge0} a slightly different notation 
($h_{i,j}$, $\alpha_{i,j}$ and $\beta_{i,j}$ instead of $u_p$, $s_{p-1}$, $t_{\pi(p)-1}$) is used to derive analogous formulas.
Lemma \ref{invexplicit} and Theorem \ref{mainresalt} are the main reason for our choice of notation in this paper: the expressions for $a_{\boldsymbol{\pi}}(\boldsymbol{u})$ and 
$b_{\boldsymbol{\pi}}(\boldsymbol{u})$, which are key in the proof of the main result, are simplest possible.
}
\end{rmk}
As pointed out in the Introduction, the first part of inequality (\ref{ineqdurbinstuart}) is known to be sharp only 
at the points
$\boldsymbol{p}_n=(-1+\frac{2}{n},-1+\frac{2}{n^2})$ with $n \geq 2$. According to \citep{Ne}, or directly using 
Lemma \ref{invexplicit}, considering 
$\boldsymbol{\pi}=(n,n-1,\ldots,2,1)$ and $u_1=u_2=\ldots=u_n=\frac{1}{n}$ we get 
$\boldsymbol{p}_n=(\tau(A_{h_{\boldsymbol{\pi},\boldsymbol{u}}}),\rho(A_{h_{\boldsymbol{\pi},\boldsymbol{u}}}))$. 
Having this, it seems natural to conjecture that all shuffles of the form
$A_{h_{\boldsymbol{\pi},\boldsymbol{u}}}$ with  
$$
\boldsymbol{\pi}=(n,n-1,\ldots,2,1), \qquad  u_1=u_2=\ldots=u_{n-1}=r,\, u_n=1-(n-1)r
$$
for some $n \geq 2$ and 
$r \in (\frac{1}{n},\frac{1}{n-1})$ might also be extremal in the sense that $(\tau(A_{h_{\boldsymbol{\pi},\boldsymbol{u}}}),\rho(A_{h_{\boldsymbol{\pi},\boldsymbol{u}}}))$ 
is a boundary point of $\Omega$. The Main content of this paper is the confirmation of this very conjecture.
We will assign all shuffles of the just mentioned form the name prototype, calculate $\tau$ and $\rho$ explicitly for all 
prototypes and then, based on these values, derive the function $\Phi$. 
\begin{defin}\label{proto}
$\boldsymbol{\pi} \in \sigma_n$ will be called \emph{decreasing} if $\boldsymbol{\pi}=(n,n-1,\ldots,2,1)$.
The pair $(\boldsymbol{\pi},\boldsymbol{u}) \in \sigma_n \times \Delta_n$ will be called a \emph{prototype} if $\boldsymbol{\pi}$ is decreasing and 
there exists some $r \in [\frac{1}{n},\frac{1}{n-1}]$ such that
$u_1=u_2=\ldots=u_{n-1}=r$ and $u_n=1-(n-1)r$.   
Analogously, $h \in \mathcal{S}^+$ (and $A_h \in \kc_d$) 
is called a \emph{prototype} if there exists a prototype $(\boldsymbol{\pi},\boldsymbol{u})$ such that $h=h_{\boldsymbol{\pi},\boldsymbol{u}}$ a.e.
\end{defin} 
Using the identities from Lemma \ref{invexplicit} we get the following expressions for prototypes (the proof is
given in the Appendix):
\begin{lem}\label{formulaproto}
Suppose that $(\boldsymbol{\pi},\boldsymbol{u})\in \sigma_n \times \Delta_n$ is a prototype, then 
\begin{eqnarray*}
\tau(A_{h_{\boldsymbol{\pi}, \boldsymbol{u}}}) &=& 1 - 4(n-1)r + 2r^2n(n-1) \in \big[\tfrac{2-n}{n},\tfrac{2-(n-1)}{n-1}\big] \\
\rho(A_{h_{\boldsymbol{\pi}, \boldsymbol{u}}}) &=& 1 - 2r(n-1)\big(3 - 3r(n-1) + r^2(n-2)n\big) \in \big[\tfrac{2-n^2}{n^2},\tfrac{2-(n-1)^2}{(n-1)^2}\big].
\end{eqnarray*}
\end{lem}
Fix $n\geq 2$. Then both functions $r \mapsto 1 - 4(n-1)r + 2r^2n(n-1)$ and 
$r \mapsto 1 - 2r(n-1)\big(3 - 3r(n-1) + r^2(n-2)n\big)$ are strictly increasing on $[\frac{1}{n},\frac{1}{n-1}]$. 
Expressing $r$ as function of $\tau$ and substituting the result in the expression for $\rho$ directly yields 
\begin{equation*}\label{Phim}
\rho(A_{h_{\boldsymbol{\pi},\boldsymbol{u}}}) = -1 - \frac{4}{n^2} + \frac{3}{n} + \frac{3\tau(A_{h_{\boldsymbol{\pi},\boldsymbol{u}}})}{n} 
  - \frac{n-2}{\sqrt{2} n^2 \sqrt{n-1} }(n - 2 + n \tau(A_{h_{\boldsymbol{\pi},\boldsymbol{u}}}))^{3/2}.\\
\end{equation*}
Based on this interrelation define $\Phi_n:[-1+\tfrac{2}{n},1] \rightarrow [-1,1]$ by
\begin{equation}\label{Phim}
\Phi_n(x)=-1 - \frac{4}{n^2} + \frac{3}{n} + \frac{3x}{n} 
  - \frac{n-2}{\sqrt{2} n^2 \sqrt{n-1} }(n - 2 + n x)^{3/2}
\end{equation}
and set 
\begin{equation}\label{Phi}
\Phi(x) = \left\{
\begin{array}{ll}
-1 & \text{if } x = -1,\\
\Phi_n(x) & \text{if } x \in  \Big[\tfrac{2-n}{n},\tfrac{2-(n-1)}{n-1}\Big] \text{ for some } n\geq 2.
\end{array} \right.
\end{equation}
Since we have $\Phi_n(\tfrac{2-n}{n})=\Phi_{n+1}(\tfrac{2-n}{n})=-1+\tfrac{2}{n^2}$ for every $n \geq 1$ this 
defines a function $\Phi:[-1,1] \rightarrow [-1,1]$. Notice that $\Phi_2(x)=-\tfrac{1}{2} + \tfrac{3x}{2}$, i.e.
on $[0,1]$ $\Phi$ coincides with Daniels' linear bound and for $x_n=\tfrac{2-n}{n}$ and $n \ge 2$ 
we have $(x_n,\Phi(x_n))=\boldsymbol{p}_n$, i.e. $(x_n,\Phi(x_n))$ coincides with the points 
at which Durbin and Stuart's inequality is known to be sharp. 
Furthermore, it is straightforward to 
verify that $\Phi$ is a strictly increasing homeomorphism on $[-1,1]$ which is concave on every interval 
$[\tfrac{2-n}{n},\tfrac{2-(n-1)}{n-1}]$ with $n \geq 2$. Figure \ref{region_locale} depicts the function $\Phi$ as
well as some prototypes and their corresponding Kendall's $\tau$ and Spearman's $\rho$.  
\begin{figure}[h!]
\centering
\makebox{\includegraphics[width=8cm]{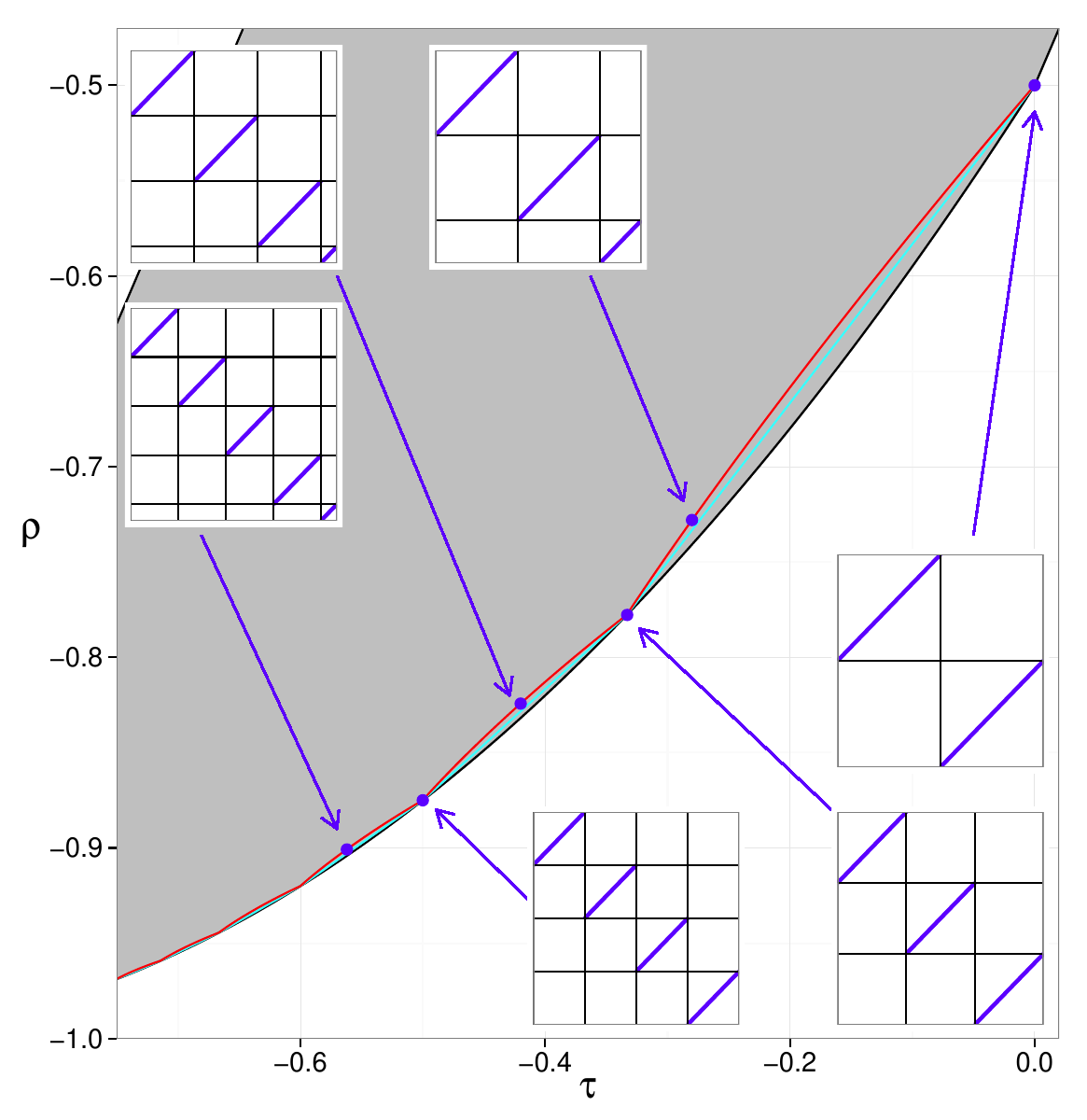}}\caption{The function $\Phi$ (red) and some prototypes with their corresponding Kendall's $\tau$ and Spearman's $\rho$. 
The shaded region depicts the classical $\tau$-$\rho$-region $\Omega_0$, 
straight lines connecting the points $p_n$ are plotted in green.}
\label{region_locale}
\end{figure}
 
Defining the compact set $\Omega_\Phi$ by 
\begin{equation}\label{omegaphi}
\Omega_\Phi=\big\{(x,y)\in [-1,1]^2: \Phi(x)\leq y \leq -\Phi(-x) \big\},
\end{equation}
we can now state the following main result the proof of which is given in the next section.
\begin{thm}\label{mainres}
The precise $\tau$-$\rho$ region $\Omega$ fulfils $\Omega \subseteq \Omega_\Phi$. 
\end{thm}
\begin{rmk}
\emph{
The fact that $\Omega \subseteq \Omega_\Phi$ holds is the principal result of this paper since it improves the
classical inequality by Durbin and Stuart mentioned in the Introduction and, more importantly, gives sharp bounds everywhere. 
In Section 5 we will, however, show that even $\Omega=\Omega_\Phi$ holds and that for every point 
$(x,y) \in \Omega$ there exists a shuffle $h \in \mathcal{S}$ such that
$(\tau(A_h),\rho(A_h))=(x,y)$.
}
\end{rmk}
\begin{rmk}
\emph{
A function similar (but not identical) to $\Phi$ has appeared in the literature in \cite{Shao}, where the authors
tried to deduce sharp bounds of $\Omega$ by running simulations (but did not provide any analytic proof). 
Additionally, it has been brought
to our attention during the preparation of this manuscript that Manuel \'Ubeda-Flores (University of Almer\'ia, Spain) already
conjectured Theorem \ref{mainres} (with the exact form of $\Phi$) in a unpublished working paper in 2009.   
}
\end{rmk}
 
\section{Proof of the main theorem}
Using the properties of $\Omega$ mentioned at the beginning of Section 3, Theorem \ref{mainres} is proved
if we can show that for every $h \in \mathcal{S}^+$ we have $\rho(A_{h}) \geq \Phi(\tau(A_{h}))$. 
Given Lemma \ref{invexplicit} it is straightforward to verify that this is equivalent to showing 
$\invs(h) \leq \varphi(\inv(h))$ for every $h \in \mathcal{S}^+$ where $\varphi:[0,\tfrac{1}{2}] \rightarrow [0,\tfrac{1}{6}]$
is defined by
\begin{equation}\label{phi}
\varphi(x) = \left\{
\begin{array}{ll}
\tfrac{1}{6} & \text{if } x = \frac{1}{2},\\
\varphi_n(x) & \text{if } x \in  [\tfrac{1}{2}-\tfrac{1}{2(n-1)},\tfrac{1}{2}-\tfrac{1}{2n}] \text{ for some } n\geq 2
\end{array} \right.
\end{equation}   
and $\varphi_n: [\tfrac{1}{2}-\tfrac{1}{2(n-1)},\tfrac{1}{2}-\tfrac{1}{2n}] \rightarrow [0,\tfrac{1}{6}]$ is given by
\begin{equation}\label{varphin}
\varphi_n(x)=\frac{1}{6} + \frac{1}{3n^2} - \frac{1}{2n} +\frac{x}{n} + \frac{n-2}{6n^2 \sqrt{n-1}}(n-1-2nx)^{3/2}.
\end{equation}
Translating this to $a_{\boldsymbol{\pi}}(\boldsymbol{u})$ and $b_{\boldsymbol{\pi}}(\boldsymbol{u})$, using eq. (\ref{defab}) and defining 
$\vartheta:[0,\tfrac{1}{2}] \rightarrow [0,\tfrac{1}{6}]$ by $\vartheta(x)=x-2\varphi(x)$
we arrive at the following equivalent form of Theorem \ref{mainres}:
\begin{thm}\label{mainresalt}
For every $n \in \mathbb{N}$, $\boldsymbol{\pi}\in \sigma_n$ and $\boldsymbol{u} \in \Delta_n$ the following inequality holds:
\begin{equation}\label{finalineq}
b_{\boldsymbol{\pi}}(\boldsymbol{u}) \geq \vartheta(a_{\boldsymbol{\pi}}(\boldsymbol{u}))
\end{equation}
\end{thm} 
We are now going to prove this result and start with some first observations and an outline of the structure of 
the subsequent proof. (i) $\vartheta$ is continuous and, by
calculating the derivative, it is straightforward to see that $\vartheta$ is non-decreasing. 
(ii) $\vartheta(0)=\vartheta(\tfrac{1}{4})=0$ and $\vartheta(\tfrac{1}{2})=\tfrac{1}{6}$. (iii) For every 
prototype $(\boldsymbol{\pi},\boldsymbol{u})$ we have the equality $b_{\boldsymbol{\pi}}(\boldsymbol{u}) = \vartheta(a_{\boldsymbol{\pi}}(\boldsymbol{u}))$. (iv) For given $n$ and fixed $\boldsymbol{\pi} \in \sigma_n$
the functions $\boldsymbol{v} \mapsto a_{\boldsymbol{\pi}}(\boldsymbol{v})$ and $\boldsymbol{v} \mapsto b_{\boldsymbol{\pi}}(\boldsymbol{v})$ are continuous on $\Delta_n$, so there exists
some $\boldsymbol{u} \in \Delta_n$ minimizing the function $\boldsymbol{v} \mapsto b_{\boldsymbol{\pi}}(\boldsymbol{v}) - \vartheta(a_{\boldsymbol{\pi}}(\boldsymbol{v}))$.
(v) For $n\leq 2$ the inequality $b_{\boldsymbol{\pi}}(\boldsymbol{u}) \geq \vartheta(a_{\boldsymbol{\pi}}(\boldsymbol{u}))$ trivially holds for every 
$\boldsymbol{\pi} \in \sigma_n$ and every $\boldsymbol{u} \in \Delta_n$, so \textbf{from now on we will only consider the case $n\geq 3$}. \\

The structure of the proof of Theorem \ref{mainresalt} is as follows:
\begin{enumerate}
\item Preliminary Step 1: We prove inequality (\ref{finalineq}) for the case of decreasing $\boldsymbol{\pi} \in \sigma_n$.
\item Preliminary Step 2: We analyze how, for fixed $\boldsymbol{\pi} \in \sigma_n$, the quantities 
$a_{\boldsymbol{\pi}}(\boldsymbol{u})$ and $b_{\boldsymbol{\pi}}(\boldsymbol{u})$ change if $\boldsymbol{u} \in  \Delta_n$ changes.
\item Induction Step 1: Assuming that the result is true for all 
      $(\boldsymbol{\pi},\boldsymbol{u}) \in \sigma_m \times \Delta_m$ with $m<n$ we prove inequality (\ref{finalineq}) for 
     $(\boldsymbol{\pi},\boldsymbol{u}) \in \sigma_n \times \Delta_n$ under the hypothesis that there either exist 
       (i) $p < q < r$ such that $\boldsymbol{\pi}(r) > \boldsymbol{\pi}(q) > \boldsymbol{\pi}(p)$ or 
       (ii) $p<q<r<s$ such that $\boldsymbol{\pi}(q) > \boldsymbol{\pi}(p) > \boldsymbol{\pi}(s) > \boldsymbol{\pi}(r)$ holds.  
\item Induction Step 2: Assuming that the result is true for all 
      $(\boldsymbol{\pi},\boldsymbol{u}) \in \sigma_m \times \Delta_m$ with $m<n$ we prove inequality (\ref{finalineq}) for 
     $(\boldsymbol{\pi},\boldsymbol{u}) \in \sigma_n \times \Delta_n$ with $\boldsymbol{\pi}$ not fulfilling the hypothesis in Induction Step I.
\end{enumerate}

\emph{Preliminary Step 1}: Consider $n\geq 3$ and $\boldsymbol{\pi}=(n,n-1,\ldots,2,1)$. 
Note that in this situation we have $e_1(\boldsymbol{u})=1, e_2(\boldsymbol{u})=a_{\boldsymbol{\pi}}(\boldsymbol{u}), e_3(\boldsymbol{u})=b_{\boldsymbol{\pi}}(\boldsymbol{u})$ for every
$\boldsymbol{u} \in \Delta_n$, where $e_i$ denotes the $i$-th elementary symmetric polynomial for $i \in \{1,2,3\}$, i.e.
$e_1(\boldsymbol{v}):=\sum_{i}v_i$, $e_2(\boldsymbol{v}):=\sum_{i<j} v_iv_j$ and $e_3(\boldsymbol{v}):=\sum_{i<j<k} v_i v_j v_k$ for every $v \in \mathbb{R}^n$.
Hence $a_{\boldsymbol{\pi}}(\boldsymbol{u})$ and $b_{\boldsymbol{\pi}}(\boldsymbol{u})$ do not change if we reorder the coordinates of $\boldsymbol{u}$.
\begin{lem}\label{step1}
Suppose that $n \geq 3, \boldsymbol{\pi}=(n,n-1,\ldots ,2,1), c_2 \in a_{\boldsymbol{\pi}}(\Delta_n)$ and that
$\boldsymbol{u} \in \Delta_n$ fulfils $b_{\boldsymbol{\pi}}(\boldsymbol{u})=\min \{b_{\boldsymbol{\pi}}(\boldsymbol{v}): \boldsymbol{v} \in \Delta_n \cap (a_{\boldsymbol{\pi}})^{-1}(\{c_2\}) \}$ as well as 
$u_1 \ge \dotsm \ge u_n \ge 0$. Then there exists $m \in \{1, \dotsc, n\}$ such that $u_i = 0$ for every 
$i >m$, and $u_1 = \dotsm = u_{m-1} \ge u_m$. 
\end{lem} 
\begin{proof} Note that the continuity of $b_{\boldsymbol{\pi}}$ and the compactness of $\Delta_n \cap (a_{\boldsymbol{\pi}})^{-1}(\{c_2\})$ imply the
existence of the minimum.  
We first prove the statement for the case $n = 3$ and suppose that $\boldsymbol{u}$ is a minimizer fulfilling $u_3 \geq u_2 \geq u_1 \geq 0$. 
Define a polynomial $f:\mathbb{R} \rightarrow \mathbb{R}$ by
\[
f(T) = (T-u_1)(T-u_2)(T-u_3) = T^3 -  T^2 + c_2 T - e_3(\boldsymbol{u}),
\]
and let $D_f$ denote the discriminant of $f$.
It is well known that $D_f > 0$ if and only if $f$ has three distinct real zeros and that in case of
$D_f \neq 0$ locally the zeros of $f$ are smooth (so in particular continuous) functions of the coefficients of $f$.

Suppose that $u_1 > u_2 > u_3 > 0$. Then $D_f > 0$.
Let $f_{\epsilon}(T) = T^3 -  T^2 + c_2 T - (e_3(\boldsymbol{u})-\epsilon)$, then for small enough values of 
$\epsilon > 0$, the polynomial $f_{\epsilon}$ has three distinct, positive real zeros: 
$u_{\epsilon,1}, u_{\epsilon,2}, u_{\epsilon,3}$.
Then $u_{\epsilon,1} + u_{\epsilon,2} + u_{\epsilon,3} = 1$ and 
$u_{\epsilon,1} u_{\epsilon,2} + u_{\epsilon,2} u_{\epsilon,3} + u_{\epsilon,3} u_{\epsilon,1} = c_2$, while 
$u_{\epsilon,1} u_{\epsilon,2} u_{\epsilon,3} = e_3(\boldsymbol{u}) - \epsilon < e_3(\boldsymbol{u})$, which is a contradiction.
So either $u_3 = 0$ or $u_1 = u_2$ or $u_1 > u_2 = u_3 > 0$.
In the first two cases we are done, so suppose that $u_1 > u_2 = u_3 > 0$.
Then $1 = u_1 + 2u_2$ and $c_2 = 2u_1 u_2 + u_2^2$.
Suppose that $u_1 \ge 4 u_2$.
Then $1 \ge 4 c_2$, so there are unique $y_1 \ge y_2 \ge 0$ such that $y_1 + y_2 = 1$ and $y_1 y_2 = c_2$. 
Let $y_3 = 0$, then considering $\boldsymbol{y}=(y_1,y_2,y_3)$ we get $e_1(\boldsymbol{y}) = 1$, $e_2(\boldsymbol{y}) = c_2$, and 
$e_3(\boldsymbol{y}) = 0 < e_3(\boldsymbol{u})$, which is a contradiction. So $4u_2 > u_1 > u_2$.
Let $y_1 = y_2 = \frac{2u_1+u_2}{3}$ and $y_3 = \frac{4u_2 - u_1}{3}$.
Then $y_1, y_2, y_3 \ge 0$, $e_1(\boldsymbol{y}) = 1$, $e_2(\boldsymbol{y}) = c_2$, and 
$e_3(\boldsymbol{y}) = \frac{1}{27}(2u_1+u_2)^2 (4u_2-u_1) = e_3(\boldsymbol{u}) - \frac{4}{27}(u_1-u_2)^3 < e_3(\boldsymbol{u})$, which is a contradiction.
This proves the claim for $n = 3$.

Suppose indirectly that the statement is false for some $n >3$.
Then there are $i<j<k$ such that $u_i > u_j \ge u_k > 0$.
Setting $\bar u_l:=\frac{u_l}{u_i+u_j+u_k}$ for every $l \in \{i,j,k\}$ obviously $\bar u_i +\bar u_j +\bar u_k=1$.
Applying the case $n=3$ to $\bar u_i, \bar u_j, \bar u_k$ yields  
$\bar y_i, \bar y_j, \bar y_k \in [0,1]$ such that 
$\bar y_i + \bar y_j + \bar y_k = \bar u_i + \bar u_j + \bar u_k$, $\bar y_i \bar y_j + \bar y_j \bar y_k + 
\bar y_k \bar y_i = \bar u_i \bar u_j + \bar u_j \bar u_k + \bar u_k \bar u_i$ and 
$\bar y_i \bar y_j \bar y_k < \bar u_i \bar u_j \bar u_k$.
Setting $y_l = u_l$ for every $l \in \{1, \dotsc, n\} \setminus \{i,j,k\}$ and $y_l = \bar y_l (u_i+u_j+u_k)$ 
for every $l \in \{i,j,k\}$ finally yields  
$e_1(\boldsymbol{y}) = e_1(\boldsymbol{u})$, $e_2(\boldsymbol{y}) = e_2(\boldsymbol{u})$ and $e_3(\boldsymbol{y}) < e_3(\boldsymbol{u})$, which is a contradiction.
\end{proof}
\begin{cor}
Suppose that $n \geq 3$ and that $\boldsymbol{\pi}=(n,n-1,\ldots,2,1)$. Then $b_{\boldsymbol{\pi}}(\boldsymbol{u}) \geq \vartheta(a_{\boldsymbol{\pi}}(\boldsymbol{u}))$ holds for
every $\boldsymbol{u} \in \Delta_n$. 
\end{cor}
\emph{Preliminary Step 2}: We investigate how, for fixed $\boldsymbol{\pi} \in \sigma_n$, the quantities 
$a_{\boldsymbol{\pi}}(\boldsymbol{u})$ and $b_{\boldsymbol{\pi}}(\boldsymbol{u})$ change if $\boldsymbol{u} \in  \Delta_n$ changes. 
To do so, temporarily extend $a_{\boldsymbol{\pi}}$ and $b_{\boldsymbol{\pi}}$ to the full $\mathbb{R}^n$ using the identities in 
Lemma \ref{invexplicit}. The following lemmata (whose proof is given in the Appendix) will be crucial in the sequel.
\begin{lem}\label{pertab}
Suppose that $n \geq 3$ and that $\boldsymbol{\delta}=(\delta_1,\ldots,\delta_n) \in \mathbb{R}^n$ fulfils $\sum_{i} \delta_i=0$. Then
for every $t \in \mathbb{R}$ the following identities hold:
\begin{align}
a_{{\boldsymbol{\pi}}}(\boldsymbol{u} + t\boldsymbol{\delta}) - a_{{\boldsymbol{\pi}}}(\boldsymbol{u}) &= \alpha_1 t + \alpha_2 t^2 \label{diffa}\\
b_{{\boldsymbol{\pi}}}(\boldsymbol{u} + t\boldsymbol{\delta}) - b_{{\boldsymbol{\pi}}}(\boldsymbol{u}) &= \beta_1 t + \beta_2 t^2 + \beta_3 t^3 \label{diffb}
\end{align}
where
\[
\alpha_1 = \sum_i a_i \delta_i, \qquad
\alpha_2 = \sum_{i < j, \, \{i,j\} \in I_{{\boldsymbol{\pi}}}} \delta_i \delta_j,
\]
\[
\beta_1 = \sum_i b_i \delta_i, \qquad
\beta_2 = \sum_{i<j} c_{i,j} \delta_i \delta_j, \qquad
\beta_3 = \sum_{i<j<k, \{i,j,k\} \in Q_{{\boldsymbol{\pi}}}} \delta_i \delta_j \delta_k,
\]
and
\[
a_i = \sum_{j: \, \{i,j\} \in I_{{\boldsymbol{\pi}}}} u_j, \qquad
b_i = \sum_{j<k, \, \{i,j,k\} \in Q_{{\boldsymbol{\pi}}}} u_j u_k, \qquad
c_{i,j} = c_{j,i} = \sum_{k: \, \{i,j,k\} \in Q_{{\boldsymbol{\pi}}}} u_k.
\]
 
\end{lem}    
\begin{lem} \label{lemma:triangle-ineq}
Suppose that $n\geq 3$ and that $\boldsymbol{\pi} \in \sigma_n$.
If $p,q,r \in \{1, \dotsc, n\}$ are distinct elements such that $\{p,q,r\} \notin Q_{\boldsymbol{\pi}}$, 
then $c_{p,r} + c_{q,r} \ge c_{p,q} \ge 0$.
\end{lem}
We now state two conditions for $\boldsymbol{\pi}$ that imply the existence of a direction 
$\boldsymbol{\delta}\in \mathbb{R}^n\setminus \{0\}$ with $\sum_{i} \delta_i=0$ 
such that $t \mapsto a_{\boldsymbol{\pi}}(\boldsymbol{u}+t\boldsymbol{\delta})-a_{\boldsymbol{\pi}}(\boldsymbol{u})$ is identical to zero for every $t$ and 
$t \mapsto b_{\boldsymbol{\pi}}(\boldsymbol{u}+t\boldsymbol{\delta})-b_{\boldsymbol{\pi}}(\boldsymbol{u})$ is of degree two and concave.  
\begin{lem}\label{sufficientprel}
Suppose that $n\geq 3$, that $\boldsymbol{\pi} \in \sigma_n$, and that one of the following two conditions holds:
\begin{enumerate}
\item[(i)] There exist $p,q,r \in \{1,2,\ldots,n\}$ with $p < q < r$ and $\boldsymbol{\pi}(r) > \boldsymbol{\pi}(q) > \boldsymbol{\pi}(p)$. 
\item[(ii)] There exist $p,q,r,s \in \{1,2,\ldots,n\}$ with $p<q<r<s$ and $\boldsymbol{\pi}(q) > \boldsymbol{\pi}(p) > \boldsymbol{\pi}(s) > \boldsymbol{\pi}(r)$.  
\end{enumerate}
Then there exists $\boldsymbol{\delta}\in \mathbb{R}^n\setminus \{0\}$  
such that the coefficients in (\ref{diffa}) and (\ref{diffb}) fulfil 
$\alpha_1=\alpha_2=\beta_3=0$ and $\beta_2 \leq 0$.
\end{lem}

\emph{Induction Step 1}: We prove the induction step for every $\boldsymbol{\pi} \in \sigma_n$ fulfilling one of the 
conditions in Lemma \ref{sufficientprel}. 
\begin{lem}
Suppose that $n \geq 3$ and that $b_{\boldsymbol{\omega}} (\boldsymbol{v}) \geq \vartheta(a_{\boldsymbol{\omega}}(\boldsymbol{v}))$ holds for all 
$(\boldsymbol{\omega},\boldsymbol{v}) \in \sigma_m \times \Delta_m$ with $m<n$. If $\boldsymbol{\pi} \in \sigma_n$ fulfils one of the 
conditions in Lemma \ref{sufficientprel} then $b_{\boldsymbol{\pi}}(\boldsymbol{u}) \geq \vartheta(a_{\boldsymbol{\pi}}(\boldsymbol{u}))$ for every $\boldsymbol{u} \in \Delta_n$.
\end{lem}
\begin{proof}
Suppose that $\boldsymbol{\pi} \in \sigma_n$ fulfils one of the conditions in Lemma \ref{sufficientprel} and consider 
$\boldsymbol{u} \in \Delta_n$. If $u_k=0$ for some $k \in \{1,\ldots,n\}$ then, defining $(\boldsymbol{\pi}',\boldsymbol{v}) \in \sigma_{n-1} \times \Delta_{n-1}$
by $v_i=u_i$ for $i<k$ and $v_i=u_{i+1}$ for $i \geq k$ as well as
$$
\boldsymbol{\pi}'(i) = \left\{
\begin{array}{ll}
\boldsymbol{\pi}(i) & \text{if } i<k \text{ and } \boldsymbol{\pi}(i)<\boldsymbol{\pi}(k)\\
\boldsymbol{\pi}(i)-1 & \text{if } i<k \text{ and } \boldsymbol{\pi}(i)>\boldsymbol{\pi}(k),\\
\boldsymbol{\pi}(i+1) & \text{if } i\geq k \text{ and } \boldsymbol{\pi}(i+1)<\boldsymbol{\pi}(k)\\
\boldsymbol{\pi}(i+1)-1 & \text{if } i \geq k\text{ and } \boldsymbol{\pi}(i+1)>\boldsymbol{\pi}(k),\\
\end{array} \right.
$$
we immediately get $b_{\boldsymbol{\pi}}(\boldsymbol{u})=b_{\boldsymbol{\pi}'}(\boldsymbol{v}) \geq \vartheta(a_{\boldsymbol{\pi}'}(\boldsymbol{v}))=\vartheta(a_{\boldsymbol{\pi}}(\boldsymbol{u}))$. 

Suppose now that $\boldsymbol{u} \in (0,1)^n$ and, using Lemma \ref{sufficientprel}, 
choose $\boldsymbol{\delta} \in \mathbb{R}^n\setminus \{0\}$ such that $\beta_2\leq 0$ and $a_{\boldsymbol{\pi}}(\boldsymbol{u}+t \boldsymbol{\delta})=a_{\boldsymbol{\pi}}(\boldsymbol{u})$ and
$b_{\boldsymbol{\pi}}(\boldsymbol{u}+t \boldsymbol{\delta})-b_{\boldsymbol{\pi}}(\boldsymbol{u})=\beta_1t + \beta_2 t^2$ for all $t \in \mathbb{R}$.
Considering $\boldsymbol{u} \in (0,1)^n$ there are $t_0 < 0 < t_1$ such that 
$\boldsymbol{u}+t\boldsymbol{\delta} \in [0,1]^n$ if and only if $t \in [t_0, t_1]$. Concavity of $t \mapsto b_{\boldsymbol{\pi}}(\boldsymbol{u} + t \boldsymbol{\delta})$ 
implies that $b_{\boldsymbol{\pi}}(\boldsymbol{u} + t_0 \boldsymbol{\delta}) \le b_{\boldsymbol{\pi}}(\boldsymbol{u})$ or $b_{\boldsymbol{\pi}}(\boldsymbol{u} + t_1 \boldsymbol{\delta}) \le b_{\boldsymbol{\pi}}(\boldsymbol{u})$.
Moreover there are $i,j$ such that $(\boldsymbol{u}+t_0 \boldsymbol{\delta})_i = 0$ and $(\boldsymbol{u}+t_1 \boldsymbol{\delta})_j = 0$ by construction, 
so we can proceed as in the first step of the proof and use induction to get $b_{\boldsymbol{\pi}}(\boldsymbol{u}) \geq \vartheta(a_{\boldsymbol{\pi}}(\boldsymbol{u}))$.
\end{proof}

\emph{Induction Step 2:} As a final step we concentrate on permutations 
$\boldsymbol{\pi} \in \sigma_n$ not fulfilling any of the two conditions in \ref{sufficientprel} and start with the following
definition and the subsequent lemma (whose proof can be found in the Appendix).
\begin{defin}
A permutation $\boldsymbol{\pi} \in \sigma_l$ is called \emph{almost decreasing} if there is at most one 
$i \in \{1, \dotsc, l-1\}$ so that $\boldsymbol{\pi}(i) < \boldsymbol{\pi}(i+1)$.  
\end{defin}
\begin{lem}\label{adequ}
Let $l \geq 1 $ and $\boldsymbol{\pi} \in \sigma_l$. Then the following two conditions are equivalent:
\begin{itemize}
\item
There are no $1 \le p < q < r \le l$ so that $\boldsymbol{\pi}(p) < \boldsymbol{\pi}(q) < \boldsymbol{\pi}(r)$, and there are no $1 \le p < q < r < s \le l$ so that $\boldsymbol{\pi}(r) < \boldsymbol{\pi}(s) < \boldsymbol{\pi}(p) < \boldsymbol{\pi}(q)$.
\item
$\boldsymbol{\pi}$ or $\boldsymbol{\pi}^{-1}$ is almost decreasing.
\end{itemize}
\end{lem}
Having this characterization we can now prove the remaining induction step for those $\boldsymbol{\pi} \in \sigma_n$ fulfilling that 
$\boldsymbol{\pi}$ or $\boldsymbol{\pi}^{-1}$ is almost decreasing. Notice that w.l.o.g. we may assume that $\boldsymbol{\pi} \in \sigma_n$ 
is almost decreasing since defining $\boldsymbol{v} \in \Delta_n$ by 
$v_i = u_{\boldsymbol{\pi}^{-1}(i)}$ for every $i \in \{1, \dotsc, n\}$ yields $a_{\boldsymbol{\pi}}(\boldsymbol{u}) = a_{\boldsymbol{\pi}^{-1}}(\boldsymbol{v})$ as well as
$b_{\boldsymbol{\pi}}(\boldsymbol{u}) = b_{\boldsymbol{\pi}^{-1}}(\boldsymbol{v})$. Both subsequent lemmata are therefore only stated and proved for almost decreasing $\boldsymbol{\pi}$.
\begin{lem}
Suppose that $n \geq 3$ and that $b_{\boldsymbol{\omega}} (\boldsymbol{v}) \geq \vartheta(a_{\boldsymbol{\omega}}(\boldsymbol{v}))$ holds for all 
$(\boldsymbol{\omega},\boldsymbol{v}) \in \sigma_m \times \Delta_m$ with $m<n$. If $\boldsymbol{\pi} \in \sigma_n$ is almost decreasing with $\boldsymbol{\pi}(1)=n$ or
$\boldsymbol{\pi}(n)=1$ then $b_{\boldsymbol{\pi}}(\boldsymbol{u}) \geq \vartheta(a_{\boldsymbol{\pi}}(\boldsymbol{u}))$ holds for every $\boldsymbol{u} \in \Delta_n$.
\end{lem}
\begin{proof}
As before we may assume $\boldsymbol{u} \in (0,1)^n$. 
Suppose that $\boldsymbol{\pi}(1) = n$. Defining $(\boldsymbol{\pi}',\boldsymbol{u}') \in \sigma_{n-1} \times \Delta_{n-1}$ by
 $\boldsymbol{\pi}'(i)=\boldsymbol{\pi}(i+1)$ and $u'_{i}=\frac{u_{i+1}}{1-u_1}$ for every $i \in \{1,\ldots,n-1\}$ and considering 
\begin{align*}
a_{\boldsymbol{\pi}'}(\boldsymbol{u}') &= \frac{1}{(1-u_1)^2} \sum_{2 \le i<j\le n: \, \{i,j\}\in I_{\boldsymbol{\pi}}} u_i u_j
\end{align*}
yields that $a_{\boldsymbol{\pi}}(\boldsymbol{u})= (1-u_1)^2 a_{\boldsymbol{\pi}'}(\boldsymbol{u}') + u_1 (1-u_1)$. Analogously, using 
\begin{align*}
b_{\boldsymbol{\pi}'}(\boldsymbol{u}') = \frac{1}{(1-u_1)^3} \sum_{2 \le i<j<k\le n: \, \{i,j,k\}\in Q_{\boldsymbol{\pi}}} u_i u_j u_k\
\end{align*}
we get $b_{\boldsymbol{\pi}}(\boldsymbol{u}) = (1-u_1)^3 b_{\boldsymbol{\pi}'}(\boldsymbol{u}') + u_1 (1-u_1)^2 a_{\boldsymbol{\pi}'}(\boldsymbol{u}')$. 
To simplify notation let $\boldsymbol{\tilde{\pi}}_k$ denote the decreasing permutation in $\sigma_k$ for every $k \in \mathbb{N}$.
Choose $u''_1,...,u''_{n-1} \in \Delta_{n-1}$ such that $a_{\boldsymbol{\tilde{\pi}}_{n-1}}(\boldsymbol{u}'')=a_{\boldsymbol{\pi}'}(\boldsymbol{u}')$ and 
$b_{\boldsymbol{\tilde{\pi}}_{n-1}}(\boldsymbol{u}'')=\vartheta(a_{\boldsymbol{\pi}'}(\boldsymbol{u}'))$. Define $\boldsymbol{\tilde{u}}=(\tilde{u}_1,...,\tilde{u}_n)$ 
by $\tilde{u}_1=u_1$ and $\tilde{u}_i=(1-u_1)u''_{i-1}$ for every $i\in \{2,\ldots,n\}$. Then
 $\sum_{i=1}^n \tilde{u}_i=u_1+(1-u_1)\sum_{i=1}^{n-1}u''_i =1$ and we get
\begin{align*}
a_{\boldsymbol{\tilde{\pi}}_n}(\boldsymbol{\tilde{u}}) &=(1-u_1)^2 a_{\boldsymbol{\tilde{\pi}}_{n-1}}(\boldsymbol{u}'') + u_1 (1-u_1) = a_{\boldsymbol{\pi}}(\boldsymbol{u})
\end{align*}
as well as
\begin{align*}
b_{\boldsymbol{\tilde{\pi}}_n}(\boldsymbol{\tilde{u}}) &=  \sum_{1<i<j<k: \, \{i,j,k\}\in Q_{\boldsymbol{\tilde{\pi}}_n}} \tilde{u}_i \tilde{u}_j 
\tilde{u}_k + \tilde{u}_1 \sum_{1<j<k: \, \{j,k\}\in I_{\boldsymbol{\tilde{\pi}}_n}} \tilde{u}_i \tilde{u}_j\\
&= (1-u_1)^3 b_{\boldsymbol{\tilde{\pi}}_{n-1}}(\boldsymbol{u}'')+ u_1 (1-u_1)^2 a_{\boldsymbol{\tilde{\pi}}_{n-1}}(\boldsymbol{u}'').
\end{align*}
Altogether this yields
\begin{align*}
b_{\boldsymbol{\pi}}(\boldsymbol{u}) &= (1-u_1)^3 b_{\boldsymbol{\pi}'}(\boldsymbol{u}') + u_1 (1-u_1)^2 a_{\boldsymbol{\pi}'}(\boldsymbol{u}')\\
&\ge (1-u_1)^3 \vartheta( a_{\boldsymbol{\pi}'}(\boldsymbol{u}')) + u_1 (1-u_1)^2 a_{\boldsymbol{\pi}'}(\boldsymbol{u}')\\
&= (1-u_1)^3 b_{\boldsymbol{\tilde{\pi}}_{n-1}}(\boldsymbol{u}'') + u_1 (1-u_1)^2 a_{\boldsymbol{\boldsymbol{\tilde{\pi}}}_{n-1}}(\boldsymbol{u}'')= b_{\boldsymbol{\tilde{\pi}}_n}(\boldsymbol{\tilde{u}}) \\
&\geq  \vartheta( a_{\boldsymbol{\tilde{\pi}}_n}(\boldsymbol{\tilde{u}}))= \vartheta(a_{\boldsymbol{\pi}}(\boldsymbol{u})).
\end{align*}
The proof of the case $\boldsymbol{\pi}(n)=1$ is completely analogous.
\end{proof}
The following final lemma assures that in case of almost decreasing $\boldsymbol{\pi} \in \sigma_n$ with 
$\boldsymbol{\pi}(1)\not=n$ and $\boldsymbol{\pi}(n)\not=1$ we cannot be on the boundary of $\Omega_\Phi$. Note that in the proof we do not make
use of the induction hypothesis.  
\begin{lem}\label{lastlem}
Suppose that $n \geq 3$ and that $\boldsymbol{\pi} \in \sigma_n$ is almost decreasing with $\boldsymbol{\pi}(1)\not=n$ and
$\boldsymbol{\pi}(n)\not=1$. Then for every $\boldsymbol{u} \in \Delta_n \cap (0,1)^n$ we have  
\begin{equation*}
b_{\boldsymbol{\pi}}(\boldsymbol{u})-\vartheta(a_{\boldsymbol{\pi}}(\boldsymbol{u})) > \min\Big\{b_{\boldsymbol{\omega}}(\boldsymbol{v})-\vartheta(a_{\boldsymbol{\omega}}(\boldsymbol{\boldsymbol{\pi}})): {\boldsymbol{\omega}} \in \sigma_n, \, \boldsymbol{v} \in \Delta_n \Big\}
\end{equation*}
\end{lem}
\begin{proof}
First note that the existence of the minimum is assured by the fact that $\sigma_n$ is finite and $\Delta_n$ is compact. 
Set $k := \boldsymbol{\pi}^{-1}(1)$. Then $1 = \boldsymbol{\pi}(k) < \boldsymbol{\pi}(k-1) < \dotsm < \boldsymbol{\pi}(1) < n$ and $1 < \boldsymbol{\pi}(n)< \dotsm < \boldsymbol{\pi}(k+2) < \boldsymbol{\pi}(k+1)$, 
so $\boldsymbol{\pi}(k+1) = n$.
Define $(\boldsymbol{\pi}',\boldsymbol{u}') \in \sigma_n \times \Delta_n$ as follows: 
$\boldsymbol{\pi}'(i) = \boldsymbol{\pi}(i)$ for $i \notin \{k, k+1\}$, $\boldsymbol{\pi}'(k) = \boldsymbol{\pi}(k+1) = n$ and 
$\boldsymbol{\pi}'(k+1) = \boldsymbol{\pi}(k) = 1$; $u'_i = u_i$ for $i \notin \{k,k+1\}$, $u'_k = u_{k+1}$ and $u'_{k+1} = u_k$.
Then it is straightforward to verify that $a_{\boldsymbol{\pi}'}(\boldsymbol{u}') - a_{\boldsymbol{\pi}}(\boldsymbol{u}) = u_k u_{k+1} $
and
\[
b_{\boldsymbol{\pi}'}(\boldsymbol{u}') - b_{\boldsymbol{\pi}}(\boldsymbol{u}) = - \sum_{i \neq k, k+1} u_k u_{k+1} u_i,
\]
holds, which, considering that $n \ge 3$ implies $a_{\boldsymbol{\pi}'}(\boldsymbol{u}') > a_{\boldsymbol{\pi}}(\boldsymbol{u})$ and $b_{\boldsymbol{\pi}'}(\boldsymbol{u}') < b_{\boldsymbol{\pi}}(\boldsymbol{u})$.
Having this we get $b_{\boldsymbol{\pi}}(\boldsymbol{u}) - \vartheta(a_{\boldsymbol{\pi}}(\boldsymbol{u})) > b_{\boldsymbol{\pi}'}(\boldsymbol{u}') - \vartheta(a_{\boldsymbol{\pi}'}(\boldsymbol{u}'))$ since 
$\vartheta$ is non-decreasing, which completes the proof.
\end{proof}
Since Lemma \ref{lastlem} implies that in order to prove inequality (\ref{finalineq}) for every $\boldsymbol{\pi} \in \sigma_n$ 
and $\boldsymbol{u} \in \Delta_n$ it is not necessary to consider almost decreasing permutations $\boldsymbol{\pi}$ with $\boldsymbol{\pi}(1)\not=n$ 
and $\boldsymbol{\pi}(n)\not=1$
the proof of Theorem \ref{mainresalt} (hence the one of Theorem \ref{mainres}) is complete.  
\section{Additional related results}
So far we have shown that $\Omega \subseteq \Omega_\Phi$. We now prove that the two sets are in fact identical.
\begin{thm}\label{identical}
The precise $\tau$-$\rho$ region $\Omega$ coincides with $\Omega_\Phi$. $\Omega$ is not convex.
\end{thm} 
\begin{proof}
The construction of $\Phi$ implies the existence of a family $(A_t)_{t \in [0,1]}$ of shuffles of $M$ fulfilling that
the map $t\mapsto A_t$ is continuous on $[0,1]$ (w.r.t. $d_\infty$) and that
\begin{equation}\label{gamma}
\gamma(t):=(\tau(A_t),\rho(A_t)) = \left\{
\begin{array}{ll}
(4t-1,\Phi(4t-1)) & \text{if } t \in [0,\tfrac{1}{2}]\\
(3-4t,-\Phi(4t-3)) & \text{if } t \in [\tfrac{1}{2},1].
\end{array} \right.
\end{equation} 
Obviously the curve $\gamma:[0,1] \rightarrow [-1,1]^2$ is simply closed and rectifiable. For every $s \in [0,1]$ consider the
similarities $f_s,g_s:[0,1]^2 \rightarrow [0,1]^2$, given by $f_s(x,y)=s(x,y)$ and $g_s(x,y)=(1-s)(x,y)+(s,s)$ and 
define the operator $O_s: \kc \rightarrow \kc$ implicitly via 
\begin{align*}
\mu_{O_s(A)}= s \mu_M^{f_s} \,+\, (1-s) \mu_A^{g_s}. 
\end{align*}  
Notice that $O_s(A)$ is usually referred to as the ordinal sum of $M,A$ with respect to the partition $[0,s),(s,1]$, see \cite{DuS}. 
Then we have $d_\infty(O_s(A),O_s(B))\leq d_\infty(A,B)$ for all $A,B \in \kc$ and every $s \in [0,1]$, and the 
mapping $s \mapsto O_s(A)$ is continuous for every $A \in \kc$. 
Consequently, the function $H:[0,1]^2 \rightarrow [-1,1]^2$, given by 
\begin{align*}
H(s,t)=\big(\tau(O_s(A_t)), \rho(O_s(A_t))\big)
\end{align*}  
is continuous and fulfils, firstly, that $H(0,t)=\gamma(t)$ and $H(1,t)=(1,1)$ for every $t \in [0,1]$ and, secondly, 
that $H(s,0)=H(s,1)$ for all $s \in [0,1]$. In other words, $H$ is a homotopy and $\gamma$ is homotopic to the
point $(1,1)$ (see Figure \ref{homotopy_plot}), implying $\Omega=\Omega_\Phi$. 
Since $\Phi$ is strictly concave on each interval 
$[\tfrac{2-n}{n},\tfrac{2-(n-1)}{n-1}]$ with $n\geq 3$, $\Omega=\Omega_\Phi$ cannot be convex.  
\end{proof}
Considering that the operator $O_s: \kc \rightarrow \kc$ maps the family of all shuffles of $M$ into itself for every
$s\in[0,1]$ the proof of Theorem \ref{identical} has the following surprising byproduct:
\begin{cor}
For every point $(x,y) \in \Omega$ there is a shuffle $h \in \mathcal{S}$ such that we have $(\tau(A_h),\rho(A_h))=(x,y)$.
\end{cor}
Additionally, Theorem \ref{identical} also implies the following result concerning the possible range of Spearman's $\rho$ 
if Kendall's $\tau$ is known (and vice versa):
\begin{cor}
Suppose that $X,Y$ are continuous random variables with $\tau(X,Y)=\tau_0$.
Then $\rho(X,Y) \in [\Phi(\tau_0),-\Phi(-\tau_0)]$. 
\end{cor}
\begin{rmk}
\emph{Due to the simple analytic form of $\Phi$ it is straightforward to verify that
\begin{align*}
\lambda_2(\Omega)=\frac{4}{5}-\frac{4}{5}\, \zeta(3)+\frac{2}{15} \pi^2 \approx 1.1543,
\end{align*}
whereby $\zeta(3)=\sum_{i=1}^\infty \tfrac{1}{i^3}$. Considering that $\lambda_2(\Omega_0)=\tfrac{7}{6}\approx 1.1667$ this underlines the 
quality of the classical inequalities. 
}
\end{rmk} 
\begin{figure}[H]
\centering
\makebox{\includegraphics[width=12cm]{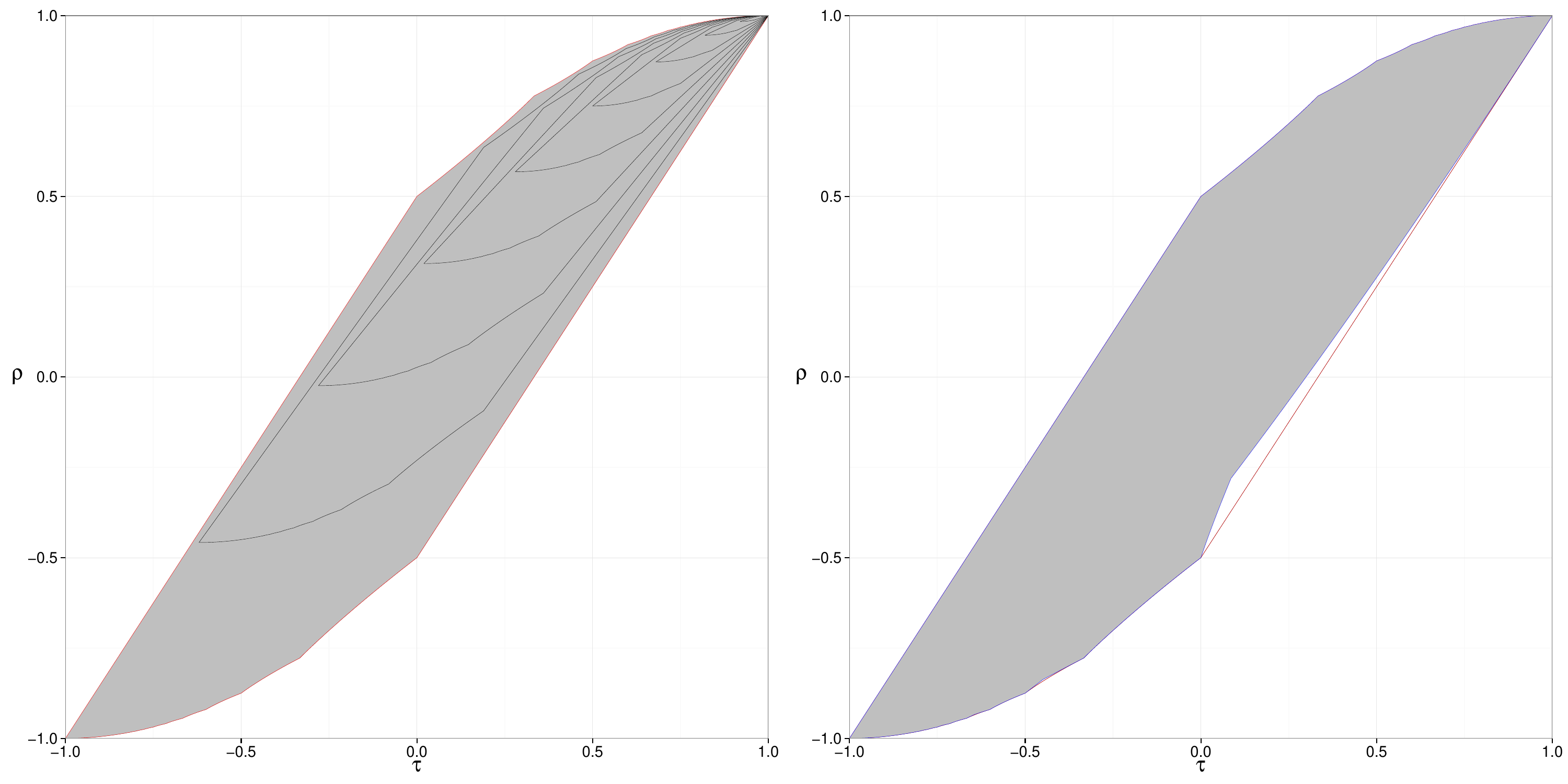}}\caption{The curves $\gamma_s(t)=H(s,t)$ for 
$t \in [0,1]$ and $s \in \{\tfrac{1}{10},\ldots,\tfrac{9}{10}\}$ with 
$H$ being the homotopy used in the proof of Theorem \ref{identical} (left panel); the region corresponding to conjecture \textbf{(C1)} for the
class of exchangeable copulas (right panel). The red lines depict the boundary of $\Omega$.}
\label{homotopy_plot}
\end{figure}

\section{Conclusions and future work}
Although Kendall's $\tau$ and Spearman's $\rho$ are both measures of concordance, they quantify different aspects of the dependence structure \citep{FreNe}.
The results established in this paper do not only answer the open question 
about how much they can differ in a definitive manner, they also show for which dependence structures the discrepancy is maximal and, more surprisingly, that 
mutually completely dependent random variables cover all possible joint values of $\tau$ and $\rho$. Although mutual complete dependence might be considered a highly atypical 
dependence structure, it naturally appears in various practical settings, in particular in optimization problems like for instance in worst-case value-at-risk scenarios \citep{Mak}. 

To the best of the authors' knowledge the combination of combinatorical and continuity arguments used in the proofs is novel. The authors conjecture that this combination may also
prove useful in various other problems - in particular concerning open questions in complete mixability, where many important recent results are 
based on combinatorical arguments and discretizations \citep{Wa}.

Having at hand a full characterization of $\Omega$ the question naturally arises, how much $\tau(A)$ and $\rho(A)$ may differ if 
it is known that $A \in  \kc$ is an element of a given subclass of copulas. 
The authors are convinced that the problem of determining the exact $\tau$-$\rho$ region for subclasses may in some cases be even more difficult than
answering the sixty year old question about $\Omega$ has been, in particular for classes where there are no simple explicit formulas for $\tau$ or $\rho$.
Based on numerous simulations and analytic calculations the authors conjecture that the following inequalities hold for the class of exchangeable copulas $\kc_{ex}$,
the class of extreme-value-copulas $\kc_{ev}$ and the class of Archimedean copulas $\kc_{ar}$:
\vspace{-4mm}
\begin{enumerate}
\item[\textbf{(C1)}] Conjectured (sharp) inequality for $A \in \kc_{ex}$: $L_{ex}(\tau) \leq \rho \leq -\Phi(-\tau)$, where\\[2mm]
\small{
\hspace{-3mm}
$L_{ex}(\tau)= \begin{cases}
   \Phi(\tau) & \text{if }\tau \in \big[\tfrac{2-n}{n}, \tfrac{2-(n-1)}{n-1}\big]\text{, $n \geq 3$ odd,}\\
   \begin{matrix}
  \min\bigg\{\frac{3 \tau}{1 + n} -\frac{2 + (n-1) n}{(1 + n)^2}  - \frac{(n-3) \left(-1 + n + (1 + n) \tau\right)^{3/2}}{2 (n+1)^2 \sqrt{n-1} },\\
  \frac{3 \tau}{n} -\frac{(n-2)^2 + n}{n^2} - \frac{(n-4) (-2 + n + n \tau)^{3/2}}{2 n^2 \sqrt{n-2} }\bigg\}
 \end{matrix}, & \text{if }\tau \in \big[\tfrac{2-n}{n}, \tfrac{2-(n-1)}{n-1}\big]\text{, $n \geq 3$ even,}\\
   \tfrac{-1 + 6\tau - 3\tau^{3/2}}{2}, & \text{if }\tau \in \big[0,\tfrac{3 - 2\sqrt{2}}{2}\big]\text{,}\\
   -\tfrac{4}{9} + \tau + \tfrac{\left(1 + 3\tau\right)^{3/2}}{18}, & \text{if }\tau \in \big[\tfrac{3 - 2\sqrt{2}}{2},1\big]\text{.}
 \end{cases} $ 
 }
 \vspace{2mm}
 \normalsize
\item[\textbf{(C2)}] Conjectured inequality for $A \in \kc_{ev}$: $L_{ev}(\tau) \leq \rho \leq R_{ev}(\tau)$, where
   $L_{ev}(\tau)=\frac{3 \tau}{2+ \tau}$ and $R_{ev}(\tau)= \frac{3 \tau}{2+ \tau^3} - \frac{1}{3}(1-\tau)^2 \tau^4$  
\item[\textbf{(C3)}] Conjectured inequality for $A \in \kc_{ar}$: $L_{ar}(\tau) \leq \rho < -\Phi(-\tau)$, where
   $L_{ar}(\tau)=\frac{7 \tau - 2\tau^3}{5} \mathbf{1}_{[0,1]}(\tau) +  \frac{31 \tau - 11\tau^3}{20} \mathbf{1}_{[-1,0)}(\tau)$
\end{enumerate}
 Notice that the set determined by the inequalities in (C2) is strictly contained in the region determined by the famous Hutchinson-Lai inequalities \citep{Hue}.
 Furthermore we remark that 
 we were able to prove that for each $x \in [-1,1]$ there exists an exchangeable copula $A \in \kc_{ex}$ with $(\tau(A),\rho(A))=(x,L_{ex}(x))$ and a 
 sequence $(A_n)_{n \in \mathbb{N}}$ of Archimedean copulas such that $(\tau(A_n),\rho(A_n)) \rightarrow (x,-\Phi(-x))$ for
 $n \rightarrow \infty$. The right-hand side of Figure \ref{homotopy_plot} and Figure \ref{conjectures} depict the corresponding regions.  

\begin{figure}[h!]   
\centering
\makebox{\includegraphics[width=12cm]{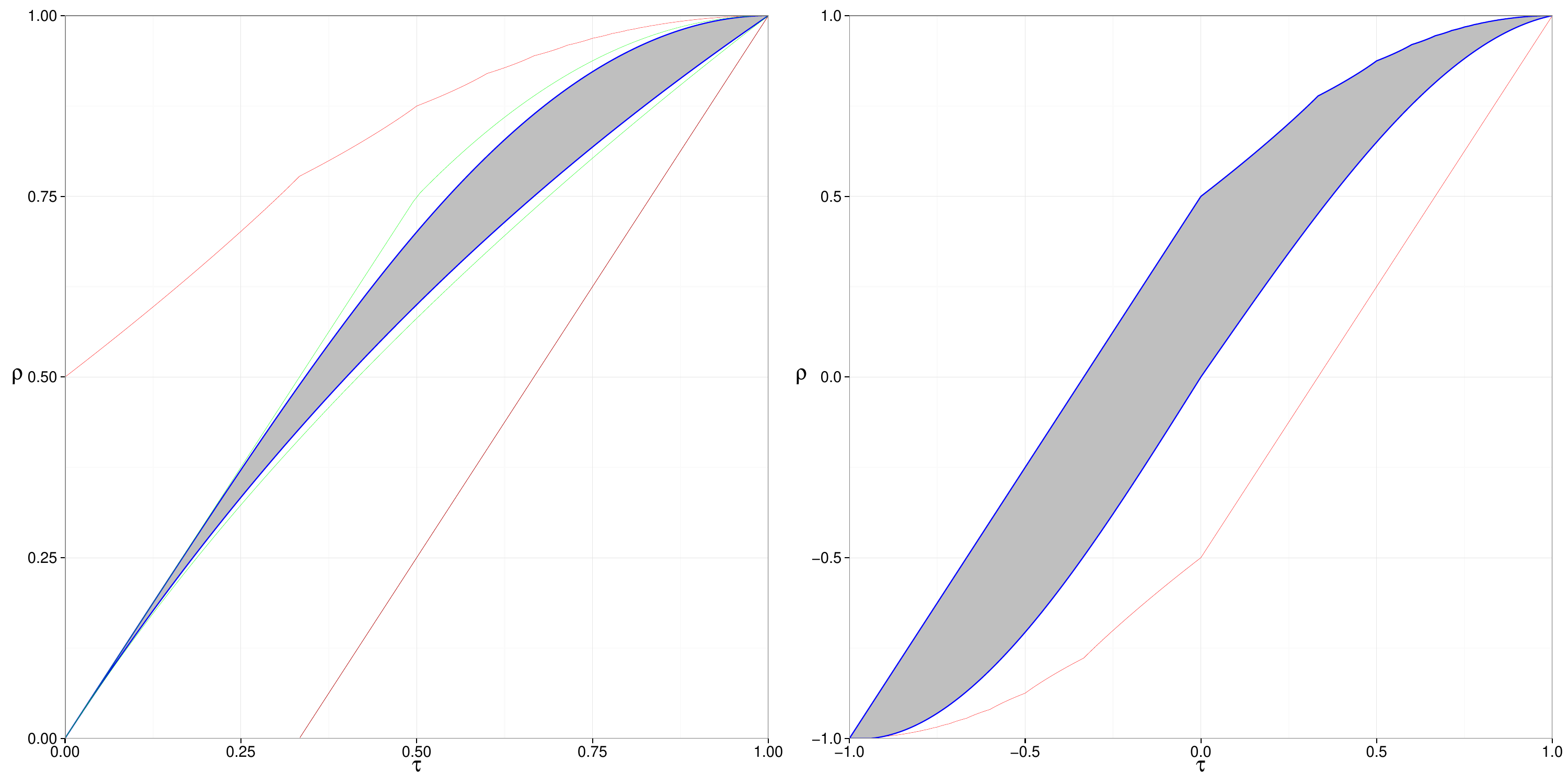}}\caption{The region corresponding to conjecture \textbf{(C2)} for the
class of extreme-value copulas (left panel) and the one corresponding to conjecture \textbf{(C3)} for the
class of Archimedean copulas (right panel); the green lines depict the Hutchinson-Lai inequalities, the red ones the boundary of $\Omega$.}
\label{conjectures}
\end{figure}
   

\section{Appendix}
\begin{proof}[of Lemma \ref{invexplicit}]
Since it is straightforward to verify the first inequality we start with the proof of the second one.
Using $s_j^2-s_{j-1}^2=(s_j-s_{j-1})(s_j+s_{j-1})=u_j(2\sum_{k<j}u_k + u_j)$ we get
\begin{align*}\allowdisplaybreaks[4]
 \invs(h_{\boldsymbol{\pi},\boldsymbol{u}}) &= \sum_{j=1}^n \int_{s_{j-1}}^{s_j}  d\lambda(y)  \int_0^1 \mathbf{1}_{[0,y)}(x) 
 \mathbf{1}_{(h_{\boldsymbol{\pi},\boldsymbol{u}}(y),1]} (h_{\boldsymbol{\pi},\boldsymbol{u}}(x))(y-x)
     \, d\lambda(x) \,\\
 &= \sum_{j=1}^n \int_{s_{j-1}}^{s_j} d\lambda(y) \sum_{i: \, i<j, \atop  \boldsymbol{\pi}(i)>\boldsymbol{\pi}(j)} \int_{s_{i-1}}^{s_{i}}  (y-x) \, d\lambda(x)
       \,  \\
 &= \sum_{j=1}^n \sum_{i: \, i<j, \atop  \boldsymbol{\pi}(i)>\boldsymbol{\pi}(j)} \int_{s_{j-1}}^{s_j}  
    \left(y \, u_i-\frac{1}{2}u_i^2 - u_i \sum_{k:\,k<i}u_k \right)  \, d\lambda(y) \\
 &= \sum_{j=1}^n \sum_{i: \, i<j, \atop  \boldsymbol{\pi}(i)>\boldsymbol{\pi}(j)} \left( \frac{1}{2} u_i u_j 
    \Big(2\sum_{l: \, l<j} u_l + u_j\Big) -\frac{1}{2} u_i^2 u_j - u_i u_j \sum_{k: \, k<i} u_k \right) \\
 &= \sum_{i<j: \atop \boldsymbol{\pi}(i)>\boldsymbol{\pi}(j)} \left( u_i u_j \sum_{l: \, l<j} u_l + 
     \frac{1}{2} u_i u_j^2 -\frac{1}{2} u_i^2 u_j - u_i u_j \sum_{k: \, k<i} u_k \right)\\
 &= \sum_{i<j: \atop \boldsymbol{\pi}(i)>\boldsymbol{\pi}(j)} \left( \frac{1}{2} u_i u_j^2  + \frac{1}{2}u_i^2 u_j + \sum_{k: \, i<k<j}u_i u_j u_k \right). 
\end{align*}
The third equality follows from
\begin{eqnarray*}
b_{\boldsymbol{\pi}}(\boldsymbol{u}) &=& \inv(h_{\boldsymbol{\pi}, \boldsymbol{u}}) - 2\invs(h_{\boldsymbol{\pi}, \boldsymbol{u}})\\
&=& \sum_{k < i < j, \, \boldsymbol{\pi}(i) > \boldsymbol{\pi}(j)} u_i u_j u_k + \sum_{i < k < j, \, \boldsymbol{\pi}(i) > \boldsymbol{\pi}(j)} u_i u_j u_k + 
   \sum_{i < j < k, \, \boldsymbol{\pi}(i) > \boldsymbol{\pi}(j)} u_i u_j u_k   \\
&\phantom{=}& + \sum_{i < j, \, \boldsymbol{\pi}(i) > \boldsymbol{\pi}(j)} u_i^2 u_j + \sum_{i < j, \, \boldsymbol{\pi}(i) > \boldsymbol{\pi}(j)} u_i u_j^2\\
&\phantom{=}& - \sum_{i < j, \, \boldsymbol{\pi}(i) > \boldsymbol{\pi}(j)} u_i^2 u_j - \sum_{i < j, \, \boldsymbol{\pi}(i) > \boldsymbol{\pi}(j)} u_i u_j^2 -  2 \sum_{i < k < j, \, \boldsymbol{\pi}(i) > \boldsymbol{\pi}(j)} u_i u_j u_k  \\
&=& \sum_{i<j<k, \, \boldsymbol{\pi}(j)>\boldsymbol{\pi}(k)} u_i u_j u_k + \sum_{i<j<k, \, \boldsymbol{\pi}(i)>\boldsymbol{\pi}(j)} u_i u_j u_k - \sum_{i<j<k, \, \boldsymbol{\pi}(i) > \boldsymbol{\pi}(k)} u_i u_j u_k \\
&=& \sum_{i<j<k, \atop  \boldsymbol{\pi}(i)>\boldsymbol{\pi}(j)>\boldsymbol{\pi}(k)} u_i u_j u_k + \sum_{i<j<k, \atop  \boldsymbol{\pi}(j)>\boldsymbol{\pi}(i)>\boldsymbol{\pi}(k)} u_i u_j u_k + \sum_{i<j<k, \atop  \boldsymbol{\pi}(j)>\boldsymbol{\pi}(k)>\boldsymbol{\pi}(i)} u_i u_j u_k\\
&\phantom{=}& +  \sum_{i<j<k,\atop  \boldsymbol{\pi}(k)>\boldsymbol{\pi}(i)>\boldsymbol{\pi}(j)} u_i u_j u_k + \sum_{i<j<k, \atop  \boldsymbol{\pi}(i)>\boldsymbol{\pi}(k)>\boldsymbol{\pi}(j)} u_i u_j u_k + \sum_{i<j<k, \atop  \boldsymbol{\pi}(i)>\boldsymbol{\pi}(j)>\boldsymbol{\pi}(k)} u_i u_j u_k\\
&\phantom{=}& - \sum_{i<j<k, \atop  \boldsymbol{\pi}(j)>\boldsymbol{\pi}(i) > \boldsymbol{\pi}(k)} u_i u_j u_k - \sum_{i<j<k, \atop  \boldsymbol{\pi}(i) > \boldsymbol{\pi}(j)>\boldsymbol{\pi}(k)} u_i u_j u_k - \sum_{i<j<k, \atop  \boldsymbol{\pi}(i) > \boldsymbol{\pi}(k)>\boldsymbol{\pi}(j)} u_i u_j u_k\\
&=& \sum_{i<j<k, \atop  \boldsymbol{\pi}(i)>\boldsymbol{\pi}(j)>\boldsymbol{\pi}(k)} u_i u_j u_k + \sum_{i<j<k, \atop  \boldsymbol{\pi}(j)>\boldsymbol{\pi}(k)>\boldsymbol{\pi}(i)} u_i u_j u_k +  \sum_{i<j<k, \atop  \boldsymbol{\pi}(k)>\boldsymbol{\pi}(i)>\boldsymbol{\pi}(j)} u_i u_j u_k. 
\end{eqnarray*}
\end{proof}

\begin{proof}[of Lemma \ref{formulaproto}]
To simplify calculations let $e_i(\boldsymbol{v})$ denote the $i$-th elementary symmetric polynomial for $i \in \{1,2,3\}$ and
$\boldsymbol{v} \in \mathbb{R}^n$, i.e.
$e_1(\boldsymbol{v})=\sum_{i}v_i$, $e_2(\boldsymbol{v})=\sum_{i<j} v_iv_j$ and $e_3(\boldsymbol{v})=\sum_{i<j<k} v_i v_j v_k$.
Using $\sum_{i} v_i^2 = e_1 (\boldsymbol{v})^2 - 2e_2 (\boldsymbol{v})$ it follows that
\begin{align*}
\inv(h_{\boldsymbol{\pi},\boldsymbol{u}}) &= e_2 (\boldsymbol{u}) = \frac{1}{2}\Big(e_1(\boldsymbol{u})^2 - \sum_{i} u_i^2 \Big) = \frac{1}{2}\Big(1 - \sum_{i} u_i^2\Big) \\
&= \frac{1}{2}\Big(1 - (n-1)r^2 - (1-(n-1)r)^2\Big)= r (n-1)-\frac{1}{2} r^2 n (n-1).
\end{align*}
Moreover, considering $r \in [\tfrac{1}{n},\tfrac{1}{n-1}]$ we get 
$\inv(h_{\boldsymbol{\pi},\boldsymbol{u}}) \in [\tfrac{1}{2} - \tfrac{1}{2(n-1)},\tfrac{1}{2} - \tfrac{1}{2n}]$,
implying $\tau(A_{h_{\boldsymbol{\pi},\boldsymbol{u}}}) \in [-1+\tfrac{2}{n},-1+ \tfrac{2}{n-1}]$, which completes the proof of the first assertion. 
Using $\sum_i v_i^3 = e_1(\boldsymbol{v})^3 - 3 e_1(\boldsymbol{v}) e_2(\boldsymbol{v}) + 3 e_3(\boldsymbol{v})$ and Lemma \ref{invexplicit} moreover it follows that
\begin{align*}
\invs(h_{\boldsymbol{\pi},\boldsymbol{u}}) &= \frac{1}{2} \inv(h_{\boldsymbol{\pi},\boldsymbol{u}}) -\frac{1}{2} \, b_{\boldsymbol{\pi}}(h) = 
 \frac{1}{2} e_2(\boldsymbol{u}) -\frac{1}{2} e_3(\boldsymbol{u}) \\
 &= \frac{1}{2} e_2(\boldsymbol{u}) -\frac{1}{2} \Big(\frac{1}{3}\sum_i u_i^3 - \frac{1}{3} e_1(\boldsymbol{u})^3 + e_1(\boldsymbol{u}) e_2(\boldsymbol{u})\Big)  \\
&= \frac{1}{2} e_2(\boldsymbol{u}) -  \frac{1}{6}\sum_i u_i^3 + \frac{1}{6}  - \frac{1}{2} e_2(\boldsymbol{u})  =
    \frac{1}{6}  -  \frac{1}{6}\sum_i u_i^3  \\
&=  \frac{1}{6}  -  \frac{1}{6} \big( (n-1)r^3 + (1-(n-1)r)^3 \big).
\end{align*}
Again considering $r \in [\tfrac{1}{n},\tfrac{1}{n-1}]$ we get 
$\invs(h_{\boldsymbol{\pi},\boldsymbol{u}}) \in [\tfrac{1}{6}-\tfrac{1}{6(n-1)^2}, \tfrac{1}{6}-\tfrac{1}{6n^2}]$,
implying $\rho(A_{h_{\boldsymbol{\pi},\boldsymbol{u}}}) \in [-1+\tfrac{2}{n^2},-1+ \tfrac{2}{(n-1)^2}]$, which completes the proof.
\end{proof} 
\begin{proof}[of Lemma \ref{pertab}]
The expression for $a_{\boldsymbol{\pi}}(\boldsymbol{u} + t\boldsymbol{\delta}) - a_{\boldsymbol{\pi}}(\boldsymbol{u})$ is easily verified:
\begin{align*}
a_{\boldsymbol{\pi}}(\boldsymbol{u} + t\boldsymbol{\delta}) - a_{\boldsymbol{\pi}}(\boldsymbol{u}) &= \sum_{i<j, \, \{i,j\} \in I_{\boldsymbol{\pi}}} ((u_i + \delta_i t)(u_j + \delta_j t) - u_i u_j) \\
&= t \sum_{i<j, \, \{i,j\} \in I_{\boldsymbol{\pi}}} \delta_j u_i +\delta_i u_j + t^2 \sum_{i<j, \, \{i,j\} \in I_{\boldsymbol{\pi}}} \delta_i \delta_j\\
&= t \left( \sum_{j<i, \, \{i,j\} \in I_{\boldsymbol{\pi}}} \delta_i u_j + \sum_{i<j, \, \{i,j\} 
\in I_{\boldsymbol{\pi}}} \delta_i u_j \right)+ t^2 \alpha_2 \\ 
&= t \sum_{i=1}^n \delta_i \sum_{j: \, \{i,j\} \in I_{\boldsymbol{\pi}}} u_j  + t^2 \alpha_2 = t \sum_{i=1}^n \delta_i a_i  + t^2 \alpha_2 =  \alpha_1 t + \alpha_2 t^2 
\end{align*}
To derive the expression for $b_{\boldsymbol{\pi}}(\boldsymbol{u} + t\boldsymbol{\delta}) - b_{\boldsymbol{\pi}}(\boldsymbol{u})$ notice that
\begin{align*}
b_{\boldsymbol{\pi}}(\boldsymbol{u} + t\boldsymbol{\delta}) - b_{\boldsymbol{\pi}}(\boldsymbol{u}) &= \sum_{i<j<k, \, \{i,j,k\} \in Q_{\boldsymbol{\pi}}} \left((u_i + \delta_i t) (u_j + \delta_j t) (u_k + \delta_k t) - u_i u_j u_k\right)\\
&= \sum_{i<j<k, \, \{i,j,k\} \in Q_{\boldsymbol{\pi}}} \delta_i \delta_j \delta_k t^3+ \delta_i \delta_j u_k t^2 + 
\delta_i \delta_k u_j t^2  \\
&\phantom{=}+ \delta_j \delta_k u_i t^2 + \delta_k u_i u_j t + \delta_j u_i u_k t + \delta_i u_j u_k t \\
&= t^3 \sum_{i<j<k, \, \{i,j,k\} \in Q_{\boldsymbol{\pi}}} \delta_i \delta_j \delta_k\\
&\phantom{=} + t^2 \sum_{i<j<k, \, \{i,j,k\} \in Q_{\boldsymbol{\pi}}} \delta_i \delta_j u_k  + \delta_i \delta_k u_j + \delta_j \delta_k u_i\\
&\phantom{=} + t \sum_{i<j<k, \, \{i,j,k\} \in Q_{\boldsymbol{\pi}}} \delta_k u_i u_j + \delta_j u_i u_k + \delta_i u_j u_k = \beta_1 t + \beta_2 t^2 + \beta_3 t^3
\end{align*}
since
\begin{align*}
\sum_{i<j<k:\atop \{i,j,k\} \in Q_{\boldsymbol{\pi}}} & \delta_i \delta_j u_k  + \delta_i \delta_k u_j + \delta_j \delta_k u_i \\
&= \sum_{i<j<k:\atop \{i,k,j\} \in Q_{\boldsymbol{\pi}}} \delta_i \delta_j u_k  +  \sum_{i<k<j:\atop \{i,k,j\} \in Q_{\boldsymbol{\pi}}} \delta_i \delta_j u_k +  \sum_{k<i<j:\atop \{k,i,j\} \in Q_{\boldsymbol{\pi}}} \delta_i \delta_j u_k \\
&= \sum_{i<j} \delta_i \delta_j \left( \sum_{i<j<k:\atop \{i,k,j\} \in Q_{\boldsymbol{\pi}}}  u_k  +  \sum_{i<k<j:\atop \{i,k,j\} \in Q_{\boldsymbol{\pi}}} u_k +  \sum_{k<i<j:\atop \{k,i,j\} \in Q_{\boldsymbol{\pi}}} u_k \right)\\
&= \sum_{i<j} \delta_i \delta_j  \sum_{k: \, \{i,k,j\} \in Q_{\boldsymbol{\pi}}}  u_k = \sum_{i<j} \delta_i \delta_j  c_{i,j} = \beta_2 
\end{align*}
and
\begin{align*}
\sum_{i<j<k:\atop \{i,j,k\} \in Q_{\boldsymbol{\pi}}} & \delta_k u_i u_j + \delta_j u_i u_k + \delta_i u_j u_k\\
&= \sum_{i<j<k:\atop \{i,j,k\} \in Q_{\boldsymbol{\pi}}}  \delta_k u_i u_j + \sum_{i<j<k:\atop \{i,j,k\} \in Q_{\boldsymbol{\pi}}} \delta_j u_i u_k + \sum_{i<j<k:\atop \{i,j,k\} \in Q_{\boldsymbol{\pi}}} \delta_i u_j u_k\\
&= \sum_{j<k<i:\atop \{i,j,k\} \in Q_{\boldsymbol{\pi}}}  \delta_i u_j u_k + \sum_{j<i<k:\atop \{i,j,k\} \in Q_{\boldsymbol{\pi}}} \delta_i u_j u_k + \sum_{i<j<k:\atop \{i,j,k\} \in Q_{\boldsymbol{\pi}}} \delta_i u_j u_k\\
&=\sum_{i=1}^n  \delta_i \left(\sum_{j<k<i:\atop \{i,j,k\} \in Q_{\boldsymbol{\pi}}}   u_j u_k + \sum_{j<i<k:\atop \{i,j,k\} \in Q_{\boldsymbol{\pi}}} u_j u_k + \sum_{i<j<k:\atop \{i,j,k\} \in Q_{\boldsymbol{\pi}}}  u_j u_k \right) \\
&=\sum_{i=1}^n  \delta_i \sum_{j<k}   u_j u_k  =\sum_{i=1}^n  \delta_i b_i =\beta_1.
\end{align*}
\end{proof}

\begin{proof}[of Lemma \ref{lemma:triangle-ineq}]
The inequality $c_{p,q} \ge 0$ immediately follows from $u_1, \dotsc, u_n \ge 0$.
Let
\[
\iota(i,j) = \iota(j,i) = \begin{cases}
1 & \textrm{if } (i<j \textrm{ and } \boldsymbol{\pi}(i) > \boldsymbol{\pi}(j)) \textrm{ or } (j<i \textrm{ and } \boldsymbol{\pi}(j) > \boldsymbol{\pi}(i)), \\
-1 & \textrm{otherwise}
\end{cases}
\]
and
\[
\gamma_{i,j,k} = \frac{1}{2}(1+\iota(i,j)\iota(i,k)\iota(j,k)) \in \{0,1\}
\]
for every $i,j,k \in \{1, \dotsc, n\}$.
Then we have $\{i,j,k\} \in Q_{\boldsymbol{\pi}}$ if and only if $\iota(i,j) \iota(j,k) \iota(i,k) = 1$ if and only if $\gamma_{i,j,k} = 1$.
Therefore $c_{i,j} = \sum_k \gamma_{i,j,k} u_k$ for every $i,j$.
So
\[
c_{p,r} + c_{q,r} - c_{p,q} = \sum_i (\gamma_{p,r,i} + \gamma_{q,r,i} - \gamma_{p,q,i}) u_i
\]
and it is enough to prove $\gamma_{p,r,i} + \gamma_{q,r,i} \ge \gamma_{p,q,i}$ for every $i$.
If $\gamma_{p,q,i} = 0$ or $\gamma_{p,r,i} = 1$ or $\gamma_{q,r,i} = 1$, then this is clear.
So suppose indirectly that $\gamma_{p,q,i} = 1$ and $\gamma_{p,r,i} = \gamma_{q,r,i} = 0$.
Then $\iota(p,q) \iota(p,i) \iota(q,i) = 1$, $\iota(p,r) \iota(p,i) \iota(r,i) = -1$ and $\iota(q,r) \iota(q,i) \iota(r,i) = -1$.
Multiplying these together, we get that $\iota(p,q) \iota(p,r) \iota(q,r) = 1$, so $\gamma_{p,q,r} = 1$.
However $\{p,q,r\} \notin Q_{\boldsymbol{\pi}}$, so $\gamma_{p,q,r} = 0$, contradiction.
\end{proof}

\begin{proof}[of Lemma \ref{sufficientprel}]
(i) Suppose that there are $p < q < r$ such that $\boldsymbol{\pi}(r) > \boldsymbol{\pi}(q) > \boldsymbol{\pi}(p)$.
Let $\delta_i = 0$ for every $i \neq p,q,r$.
We can fix a nonzero solution $(\delta_p,\delta_q,\delta_r) \in \mathbb{R}^3 \setminus \{0\}$ to the 
following system of homogeneous linear equations:
\[
\delta_p +\delta_q +\delta_r = 0, \qquad a_p\delta_p + a_q\delta_q + a_r\delta_r = 0.
\]
Then $\sum_i\delta_i = 0$, $\alpha_1 = 0$, moreover $\{p,q\}, \{p,r\}, \{q,r\} \notin I_{\boldsymbol{\pi}}$ and 
$\{p,q,r\} \notin Q_{\boldsymbol{\pi}}$, so $\alpha_2 = \beta_3 = 0$.
The numbers $\delta_p\delta_q$, $\delta_p\delta_r$, $\delta_q\delta_r$ cannot be all negative, so e.g., 
$\delta_p\delta_q \ge 0$.
Then using Lemma \ref{lemma:triangle-ineq} we obtain
\begin{align*}
\beta_2 &= c_{p,q}\delta_p\delta_q + c_{p,r}\delta_p\delta_r + c_{q,r}\delta_q\delta_r 
   \le (c_{p,r} + c_{q,r})\delta_p\delta_q + c_{p,r}\delta_p\delta_r + c_{q,r}\delta_q\delta_r \\
&= c_{p,r}\delta_p (\delta_q +\delta_r) + c_{q,r}\delta_q (\delta_p +\delta_r) = 
-c_{p,r}\delta_p^2 - c_{q,r}\delta_q^2 \le 0.
\end{align*}

Now suppose that there are $p<q<r<s$ such that $\boldsymbol{\pi}(q) > \boldsymbol{\pi}(p) > \boldsymbol{\pi}(s) > \boldsymbol{\pi}(r)$.
Let $\delta_i = 0$ for $i \neq p,q,r,s$.
We can fix a nonzero solution $(\delta_p,\delta_q,\delta_r,\delta_s) \in \mathbb{R}^4 \setminus \{0\}$ to 
the following system of homogeneous linear equations:
\[
\delta_p +\delta_q = 0, \qquad\delta_r +\delta_s = 0, \qquad a_p\delta_p + a_q\delta_q + a_r\delta_r + a_s\delta_s = 0.
\]
Then $\sum_i\delta_i = 0$, $\alpha_1 = 0$, and $\alpha_2 =\delta_p\delta_r +\delta_p\delta_s +\delta_q\delta_r +\delta_q\delta_s = (\delta_p +\delta_q)(\delta_r +\delta_s) = 0$.
Moreover $\{p,q,r\}, \{p,q,s\}, \{p,r,s\}, \{q,r,s\} \notin Q_{\boldsymbol{\pi}}$, so $\beta_3 = 0$.
We claim that
\[
\beta_2 = -c_{p,q}\delta_p^2 + (c_{p,r} + c_{q,s} - c_{p,s} - c_{q,r})\delta_p\delta_r - c_{r,s}\delta_r^2 \le 0.
\]
Let $d = c_{p,r} + c_{q,s} - c_{p,s} - c_{q,r}$.
Lemma \ref{lemma:triangle-ineq} implies that $c_{p,q} \ge |c_{q,r}-c_{p,r}|$ and $c_{p,q} \ge |c_{p,s}-c_{q,s}|$, so $2c_{p,q} \ge |c_{q,r}-c_{p,r}| + |c_{p,s}-c_{q,s}| \ge |d|$.
Similarly, $c_{r,s} \ge |c_{p,r} - c_{p,s}|$ and $c_{r,s} \ge |c_{q,s}-c_{q,r}|$, so $2c_{r,s} \ge |d|$ too.
Since either $-d \le 0$ or $d \le 0$, the equations
\begin{align*}
2 \beta_2 &= -2c_{p,q}\delta_p^2 + 2d\delta_p\delta_r - 2c_{r,s}\delta_r^2 = -d(\delta_p -\delta_r)^2 - (2c_{p,q}-d)\delta_p^2 - (2c_{r,s}-d)\delta_r^2 \\
&= d(\delta_p +\delta_r)^2 - (2c_{p,q}+d)\delta_p^2 - (2c_{r,s}+d)\delta_r^2
\end{align*}
imply that $\beta_2 \le 0$.
\end{proof}

\begin{proof}[of Lemma \ref{adequ}]
Each of the two conditions is true for $\boldsymbol{\pi}$ if and only if it is true for $\boldsymbol{\pi}^{-1}$.
It is easy to see that the second condition implies the first one.
Conversely, suppose that $\boldsymbol{\pi}$ (and hence also $\boldsymbol{\pi}^{-1}$) satisfies the first condition.
We prove by induction on $l$. The statement is trivial for $l = 1$, so let $l \ge 2$.
If $\boldsymbol{\pi}(l) = 1$ then we can use the induction hypothesis for $\boldsymbol{\pi}|_{\{1, \dotsc, l-1\}} \in \sigma_{l-1}$.
If $\boldsymbol{\pi}(1) = l$ then we can use the induction hypothesis for $\boldsymbol{\pi}' \in \sigma_{l-1}$, where $\boldsymbol{\pi}'(i) = \boldsymbol{\pi}(i+1)-1$ 
for every $i \in \{1, \dotsc, l-1\}$.
So we may assume $\boldsymbol{\pi}(1) \neq l$ and $\boldsymbol{\pi}(l) \neq 1$.

Suppose that $\boldsymbol{\pi}^{-1}(l) > \boldsymbol{\pi}^{-1}(1)$ and set $k = \boldsymbol{\pi}^{-1}(l)$.
If $\boldsymbol{\pi}^{-1}(1) < k-1 < k$, then $1 = \boldsymbol{\pi}(\boldsymbol{\pi}^{-1}(1)) < \boldsymbol{\pi}(k-1) < \boldsymbol{\pi}(k) = l$, which contradicts the condition on $\boldsymbol{\pi}$.
Consider $\boldsymbol{\pi}^{-1}(1) = k-1$. If $i < j < k$, then we cannot have $\boldsymbol{\pi}(i) < \boldsymbol{\pi}(j)$, because then we would have 
$\boldsymbol{\pi}(i) < \boldsymbol{\pi}(j) < \boldsymbol{\pi}(k) = l$, contradicting the condition on $\boldsymbol{\pi}$.
If $k-1 < i < j$ then we cannot have $\boldsymbol{\pi}(i) < \boldsymbol{\pi}(j)$, because then we would have 
$1 = \boldsymbol{\pi}(\boldsymbol{\pi}^{-1}(1))=\boldsymbol{\pi}(k-1) < \boldsymbol{\pi}(i) < \boldsymbol{\pi}(j)$, contradicting the condition on $\boldsymbol{\pi}$.
So $\boldsymbol{\pi}(i) > \boldsymbol{\pi}(i+1)$ for every $i \in \{1, \dotsc, l-1\} \setminus \{k\}$, hence $\boldsymbol{\pi}$ is almost decreasing.

Now suppose that $\boldsymbol{\pi}^{-1}(l) < \boldsymbol{\pi}^{-1}(1)$.
If $\boldsymbol{\pi}(l) < \boldsymbol{\pi}(1)$, then the condition on $\boldsymbol{\pi}$ is false for $p = 1$, $q= \boldsymbol{\pi}^{-1}(l)$, $r = \boldsymbol{\pi}^{-1}(1)$, $s = l$.
So $\boldsymbol{\pi}(l) > \boldsymbol{\pi}(1)$ and $(\boldsymbol{\pi}^{-1})^{-1}(l) > (\boldsymbol{\pi}^{-1})^{-1}(1)$. Applying the previous paragraph to 
$\boldsymbol{\pi}^{-1}$ shows that $\boldsymbol{\pi}^{-1}$ is almost decreasing.
\end{proof}






\end{document}